\documentclass[12pt,oneside,reqno]{amsart}
\usepackage{amsmath,latexsym,amssymb,amsthm,amsfonts,stmaryrd,mathrsfs,enumerate,graphicx}
\newcommand{\be}{\begin{eqnarray}}
\newcommand{\ee}{\end{eqnarray}}
\newcommand{\ce}{\begin{eqnarray*}}
\newcommand{\de}{\end{eqnarray*}}
\newtheorem{theorem}{Theorem}[section]
\newtheorem{lemma}[theorem]{Lemma}
\newtheorem{remark}[theorem]{Remark}
\newtheorem{definition}[theorem]{Definition}
\newtheorem{proposition}[theorem]{Proposition}
\newtheorem{Examples}[theorem]{Examples}
\newtheorem{corollary}[theorem]{Corollary}

\def\[{{\Big[}}
\def\]{{\Big]}}
\def\<{{\langle}}
\def\>{{\rangle}}
\def\({{\Big(}}
\def\){{\Big)}}

\def\bx{{\mathbf{x}}}

\def\bt{\begin{theorem}}
\def\et{\end{theorem}}
\def\bl{\begin{lemma}}
\def\el{\end{lemma}}
\def\br{\begin{remark}}
\def\er{\end{remark}}
\def\bx{\begin{Examples}}
\def\ex{\end{Examples}}
\def\bd{\begin{definition}}
\def\ed{\end{definition}}
\def\bp{\begin{proposition}}
\def\ep{\end{proposition}}
\def\bc{\begin{corollary}}
\def\ec{\end{corollary}}

\def\mR{{\mathbb R}}

\def\geq{\geqslant}

\allowdisplaybreaks

\begin{document}
\title{Localization of Wiener Functionals of Fractional Regularity and Applications}
\author{Kai He$^{1}$, Jiagang Ren$^{2}$, Hua Zhang$^{3}$}
\date{}
\dedicatory{$^{1}$Institute of Applied Mathematics, Academy of Mathematics\\
and System Sciences, Chinese Academy of Sciences, \\
Beijing 100190, P.R.China\\
$^{2}$School of Mathematics and Computational Science,
Sun Yat-Sen University,\\
Guangzhou, Guangdong 510275, P.R.China\\
$^{3}$School of Statistics, Jiangxi University of Finance and Economics,\\
Nanchang, Jiangxi 330013, P.R.China\\
Emails: K. He: hekai0@amss.ac.cn\\
J. Ren: renjg@mail.sysu.edu.cn\\
H. Zhang: zh860801@163.com}
\thanks{This work is supported by NSF of China (No. 11171358), Doctor Fund of Ministry of Education (no. 20100171110038) and the Key Laboratory of Random Complex Structures and Data,
Academy of Mathematics and Systems Science, Chinese Academy of Sciences.}

\begin{abstract}
In this paper we localize some of Watanabe's results on Wiener functionals of fractional regularity, and use them to give a precise estimate of the difference between two Donsker's delta functionals even with fractional differentiability. As an application, the convergence rate of the density of the Euler scheme for non-Markovian stochastic differential equations is obtained.
\end{abstract}

\maketitle

\section{Introduction}
During the last years the analysis on Wiener functionals with fractional smoothness in the sense of Malliavin calculus has drawn increasing attention. There are two ways to define fractional order Sobolev spaces as intermediate spaces between the Sobolev spaces with integer differential index. One way is
the complex interpolation method. It makes use of fractional powers of the Ornstein-Uhlenbeck operator, and we denote these fractional order Sobolev spaces by ${\mathbb D}_{\alpha}^p$, $p>1$, $\alpha\in\mR$. The spaces ${\mathbb D}_{\alpha}^p$ are natural and typical ones which corresponds to Bessel potential spaces in classical analysis. However, it is not easy to see that the space ${\mathbb D}_{\alpha}^p$ with $0<\alpha<1$ is invariant under the composition with Lipschitz functions. To circumvent this difficulty, Watanabe \cite{W} introduced real interpolation fractional order Sobolev spaces on Wiener spaces by using the trace method. Then an equivalent method, the $K$-method, was used in Airault, Malliavin and Ren \cite{AMR} to study the
smoothness of stopping times of diffusion processes. The advantage of the $K$-method is that it describes explicitly how well one can approximate a fractionally smooth Wiener functional by a sequence of smooth functionals. For the equivalence between the trace method and the $K$-method, we refer to \cite[Chap. 1]{T}. Also it should be mentioned that later Hirsch \cite{H} proved that the complex interpolation fractional order Sobolev spaces are in fact invariant under the composition with Lipschitz functions, too.

The aims of the present paper are, roughly speaking, to study local versions of some of the results of \cite{W} and to investigate their applications in the Euler scheme of non-Markovian stochastic differential equations. In particular, we establish a precise estimate of the difference between two Donsker's delta functionals in terms of fractional order Sobolev norms, and, as a consequence, we then dominate the difference between two conditional expectations in the sense of H\"{o}lder norms.

The Euler scheme is an useful tool in the numerical simulation of solutions of stochastic differential equations, and has theoretical value as well. Let $(X(\cdot))$ be the unique solution to
\begin{eqnarray}\label{SDE1}
X(t)=x+\int_0^tb(s,X(s))ds+\int_0^t\sigma(s,X(s))dW(s),
\end{eqnarray}
where $b$, $\sigma$ are respectively Lipschtiz continuous mapping from ${\mathbb R}^d$ to ${\mathbb R}^d$ and ${\mathbb R}^d\otimes{\mathbb R}^m$, and $(W(\cdot))$ is an $m$-dimensional Brownian motion. Let $T>0$ be a fixed time horizon, and $T/n$ represent the discretization step. Set $X_n(0)=x$, and for $kT/n<t\leqslant(k+1)T/n$, the Euler scheme is defined by
\begin{eqnarray*}
X_n(t)=X_n(\frac{kT}{n})+b(\frac{kT}{n},X_n(\frac{kT}{n}))(t-\frac{kT}{n})+\sigma(\frac{kT}{n},X_n(\frac{kT}{n}))
(W(t)-W(\frac{kT}{n})).
\end{eqnarray*}
There are two kinds of weak approximations. The first one concerns
\begin{eqnarray*}
\xi_1(x,T,n):=E[f(X(T))]-E[f(X_n(T))],
\end{eqnarray*}
where $f$ is a suitable class of test function. The second one is the approximation of the density $p_{X(T)}$ of the law of $X(T)$, i.e.,
\begin{eqnarray*}
\xi_2(x,T,n):=p_{X(T)}-p_{X_n(T)}.
\end{eqnarray*}

When studying these two kinds of quantities,  people's interest focuses on the convergence rate or an error expansion of $\xi_1(x,T,n)$ and $\xi_2(x,T,n)$ in terms of $T/n$, due to the fact that analysis of these two kinds of quantities turns out to be more important for applications, for instance in finance and biology, etc.  There have been a lot of progresses in this area. Suppose that the test function $f$ and the coefficients $b$ and $\sigma$ are sufficiently smooth and $f$ has polynomial growth. Without any additional assumption on the generator, Talay and Tubaro \cite{TT} derive an error expansion of order $1$ for $\xi_1(x,T,n)$. Bally and Talay \cite{BT1} also obtain the same kind of result for bounded Borel functions $f$ under the hypoellipticity assumption on the coefficients. These authors also extend their results to $\xi_2(x,T,n)$ for a slightly modified Euler scheme in \cite{BT2}. It is also worth noting that Kohatsu-Higa and Pettersson in \cite{KP} introduce another way to prove weak error expansion of $\xi_1(x,T,n)$ and $\xi_2(x,T,n)$ which is based on the integration by parts formula of the Malliavin calculus. On the other hand, by using the fractional calculus in the Malliavin calculus, Watanabe and the second named author of the present paper \cite{RW} obtained the convergence of $\xi_2(x,T,n)$ in fractional order Sobolev spaces. All these works are confined in the context of Markovian SDEs.

Establishing the estimate of the difference between two Donsker's delta functionals enables us to  study the Euler scheme of non-Markovian stochastic differential equations. In other words, we allow the coefficients in stochastic differential equations to look into the past. More precisely, we consider the solution to the equation of the form:
\begin{eqnarray}\label{SDE2}
X(t)=x+\int_0^tb(s,X(\cdot))ds+\int_0^t\sigma(s,X(\cdot))dW(s),
\end{eqnarray}
where $\sigma:[0,\infty]\times C([0,\infty];{\mathbb R}^d)\rightarrow{\mathbb R}^d\otimes{\mathbb R}^m$ and $b:[0,\infty]\times C([0,\infty];{\mathbb R}^d)\rightarrow{\mathbb R}^d$.

Let us describe our ideas explicitly as follows. Since the coefficients $(\sigma,b)$ depend not only on the present values of the solution processes, but also on its previous values too, the analysis of convergence rate or error expansion of $\xi_1(x,T,n)$ and $\xi_2(x,T,n)$ for SDE (\ref{SDE2}) is quite different to that for SDE (\ref{SDE1}). In particular, it seems difficult to extend the results in \cite{BT1} \cite{BT2} to stochastic differential equations which are not Markovian SDEs since the approaches used there rely heavily on the Feynman-Kac partial differential equations associated with SDE (\ref{SDE1}). Here unlike the approaches in \cite{BT1} \cite{BT2}, by using the fact that the heat kernel can be given by the generalized expectation of Donsker's delta functionals, we will establish the convergence rate of $\xi_2(x,T,n)$ for SDE (\ref{SDE2}). Of course, since the coefficients in (\ref{SDE2}) may depend on the past trajectories $\{X_s,0\leqslant s\leqslant t\}$ of the solution, we should modify the Euler scheme for non-Markovian stochastic differential equations which will be carried out in Section 5. In order to derive a convergence rate of $\xi_2(x,T,n)$, the strong approximation of the Euler scheme for non-Markovian stochastic differential equations in Sobolev spaces of appropriate orders in the Malliavin calculus sense is needed.

The paper is organized as follows. In Section 2, we recall some results of the Malliavin calculus that we will use in the sequel. In Section 3, we study some properties of the fractional order Sobolev spaces under local assumptions. In Section 4, we give a precise estimate of the difference between two Donsker's functionals. In Section 5, we are devoted to the proof of the convergence rate of $\xi_2(x,T,n)$.

\section{Recalls on the Malliavin calculus}
We first recall briefly some basic ingredients in the Malliavin calculus and the reader is referred, e.g., to \cite{HY, PM, N} for more details. Let $(B,\mathbb H,\mu)$ be an abstract Wiener space. We will denote the gradient operator (or Shigekawa's $\mathbb H$-derivative) by $D$, its dual divergence operator (or the Skorohod operator) by $D^*$ and the Ornstein-Uhlenbeck operator by $L:=-D^* D$. Let $E$ be a real separable Hilbert space. The Sobolev spaces ${\mathbb D}_{\alpha}^p(E)$, $1<p<\infty$, $\alpha\in{\mathbb R}$, of $E$-valued Wiener and generalized Wiener functionals are defined by
\begin{eqnarray*}
{\mathbb D}_{\alpha}^{p}(E)=(1-L)^{-\alpha/2}(\mathbb{L}^p(E))
\end{eqnarray*}
with the norm
\begin{eqnarray*}
\|F\|_{\alpha,p}=\|(1-L)^{\alpha/2}F\|_{\mathbb{L}^p(E)},
\end{eqnarray*}
where $\mathbb{L}^p(E)$ is the usual $\mathbb{L}^p(E)$-space. We denote by $\mathbb{L}^{\infty-}(E)=\bigcap_{1<p<\infty}\mathbb{L}^p(E)$,
${\mathbb D}_{\alpha}^{\infty-}(E)=\bigcap_{1<p<\infty}{\mathbb D}_{\alpha}^p(E)$, and ${\mathbb D}_{\infty}^{\infty-}(E)=\bigcap_{\alpha>0}\bigcap_{1<p<\infty}{\mathbb D}_{\alpha}^p(E)$. If $E=\mathbb{R}$, we simply write $\mathbb{L}^p$, $\mathbb{D}_{\alpha}^p$, $\mathbb{L}^{\infty-}$, ${\mathbb D}_{\alpha}^{\infty-}$ and ${\mathbb D}_{\infty}^{\infty-}$.

We denote by $C_p^{\infty}({\mathbb R}^d)$  the set of all infinitely continuously differential functions $f:{\mathbb R}^d\rightarrow{\mathbb R}$ such that $f$ and all of its partial derivatives have polynomial growth, and we also denote by $C_b^{\infty}({\mathbb R}^d)$  the set of all infinitely continuously differential functions $f:{\mathbb R}^d\rightarrow{\mathbb R}$ such that $f$ and all of its partial derivatives are bounded. Let
$C^\beta({\mathbb R}^d)$, $\beta\geqslant 0$, be the Banach space of $[\beta]$-times continuously differentiable functions
on ${\mathbb R}^d$ whose $[\beta]$-th derivatives are uniformly $\{\beta\}$-H\"{o}lder continuous
with the norm ($[\beta]$ and $\{\beta\}=\beta-[\beta]$ denote the integer and fractional part of $\beta$ respectively)
\begin{eqnarray*}
\|f\|_{C^\beta({\mathbb R}^d)}=\sum_{n;|n|\leqslant[\beta]}
|\partial^n f|_{\infty}+\sum_{n;|n|=[\beta]}\sup_{x\neq y}|\partial^n f(x)
-\partial^n f(y)|/|x-y|^{\{\beta\}},
\end{eqnarray*}
with the notation $n=(n_1,\cdots,n_d)$, $|n|=\sum_{i=1}^dn_i$ and
$\partial^n=\partial_1^{n_1}\cdots\partial_d^{n_d}$, where $\partial_i=\partial/\partial x_i$.
Note that $(1-\Delta)^{-\beta/2}(C^0({\mathbb R}^d))\subset C^\beta({\mathbb R}^d)$ (see \cite{S}).
Let $S({\mathbb R}^d)$ be the real space of rapidly decreasing $C^{\infty}$-functions.

In the Malliavin calculus, a key role is played by the Malliavin covariance matrix which is defined as follows.
\begin{definition}
Suppose that $F=(F^1,\cdots,F^d)$ is a random vector whose components belong to the space ${\mathbb D}_1^{\infty-}$. We associate to $F$ the following random symmetric nonnegative definite matrix:
\begin{eqnarray*}
\Sigma_F(\omega)=(\sigma_F^{ij}(\omega))_{1\leqslant i,j\leqslant d}:=(( DF^i,DF^j)_{\mathbb H})_{1\leqslant i,j\leqslant d}.
\end{eqnarray*}
The matrix $\Sigma_F(\omega)$ will be called the Malliavin covariance matrix of the random vector $F$. We will say that a random vector $F=(F^1,\cdots,F^d)$ whose components are in ${\mathbb D}_{\infty}^{\infty-}$ is nondegenerate if its Malliavin covariance matrix $\Sigma_F(\omega)$ is invertible a.s. and
\begin{eqnarray*}
\Gamma_F(\omega)=(\gamma_F^{ij}(\omega))_{1\leqslant i,j\leqslant d}:=(\Sigma_F(\omega))^{-1}\in
\mathbb{L}^{\infty-}({\mathbb R}^d\otimes{\mathbb R}^d).
\end{eqnarray*}
\end{definition}

Set $d\mu_G:=G\cdot d\mu$ and
\begin{eqnarray}\label{density}
p_{F,G}(y)=\frac{\mu_G(F\in dy)}{dy}.
\end{eqnarray}
Denote by $E_G$ the integral w.r.t $\mu_G$. The following proposition is taken from \cite[Proposition 23]{BC1} or \cite[Lemma 2.1]{BC2}.
\begin{proposition}\label{bally}
Let $F\in{\mathbb D}_2^{\infty-}({\mathbb R}^d)$ and let $G\in{\mathbb D}_1^{\infty-}$ take values on $[0,1]$ with
\begin{eqnarray*}
1+E_G[\|D\ln G\|_{\mathbb H}^p]<\infty\quad\text{for every}\quad p\geq 1.
\end{eqnarray*}
Assume that $A$ is a measurable set such that $G1_A=0$ and for any $p>1$,
\begin{eqnarray*}
E[|\det(\Sigma_F)|^{-p}1_{A^c}]<\infty.
\end{eqnarray*}
Then the law of $F$ under $\mu_G$ is absolutely continuous with respect to the Lebesgue measure on ${\mathbb R}^d$. Moreover, for every $p>d$ there exist some universal constants $C$ and $q>1$ depending on $d$ and $p$ such that the density $p_{F,G}$ satisfies
\begin{eqnarray*}
p_{F,G}(y)\leqslant C(1+E_G[|\det(\Sigma_F)|^{-p}])^q (1+\|F\|_{2,p,G}+\|LF\|_{p,G})^q(1+E_G[\|D\ln G\|_{\mathbb H}^p])^q.
\end{eqnarray*}
\end{proposition}

In the rest of this article, we will adopt the following notations. Let $\Psi:[0,\infty)\mapsto{\mathbb R}$ be a $C_b^{\infty}$ function ($\Psi$ and all of its partial derivatives are bounded) such that
\begin{eqnarray*}
1_{[0,\frac{1}{8}]}\leqslant\Psi\leqslant1_{[0,\frac{1}{4}]},
\end{eqnarray*}
and $\Psi_1:[0,\infty)\mapsto{\mathbb R}$ be a $C_b^{\infty}$ function such that
\begin{eqnarray*}
1_{[0,\frac{1}{4}]}\leqslant\Psi_1\leqslant 1_{[0,\frac{1}{2}]}
\end{eqnarray*}
and
\begin{eqnarray}
\sup_x|(\ln\Psi_1(x))'|^p\Psi_1(x)<\infty\quad\text{for every }p\geqslant 1. \label{bally example}
\end{eqnarray}
For $F_1,F_2\in{\mathbb D}_{\infty}^{\infty-}({\mathbb R}^d)$, we define $R_{F_1,F_2}$ by
\begin{eqnarray*}
R_{F_1,F_2}=\frac{\|D(F_1-F_2)\|_{\mathbb H}^2(1+\|\Sigma_{F_1}\|_2^2)^{(d-1)/2}}{\det(\Sigma_{F_1})},
\end{eqnarray*}
where $\|DF_i\|_{\mathbb H}^2=\sum_{j=1}^d\|DF_i^j\|_{\mathbb H}^2$, $i=1,2$ and $\|\Sigma_{F_1}\|_2$ is the Hilbert-Schmidt norm of the Malliavin covariance matrix $\Sigma_{F_1}$. It is obvious that $\Psi_1(R_{F_1,F_2})=1$ on the set $\{\Psi(R_{F_1,F_2})\neq 0\}$. For the set $\bigcup_{k=0}^{\infty}\{D^k(\Psi(R_{F_1,F_2}))\neq 0\}$ and $\bigcup_{k=0}^{\infty}\{D^k(\Psi_1(R_{F_1,F_2}))$
$\neq 0\}$, we have the following result:
\begin{eqnarray}\label{important result}
\bigcup_{k=0}^{\infty}\{D^k(\Psi(R_{F_1,F_2}))\neq 0\}\subset\{\det(\Sigma_{F_2+t(F_1-F_2)})\geqslant4^{-d}
\frac{(\det(\Sigma_{F_1}))^d}{\|\Sigma_{F_1}\|^{d(d-1)}}\}\quad\text{a.s.},
\end{eqnarray}
and
\begin{eqnarray}\label{important result1}
\bigcup_{k=0}^{\infty}\{D^k(\Psi_1(R_{F_1,F_2}))\neq 0\}\subset\{\det(\Sigma_{F_2+t(F_1-F_2)})\geqslant
(1-\frac{\sqrt{2}}{2})^{2d}\frac{(\det(\Sigma_{F_1}))^d}{\|\Sigma_{F_1}\|^{d(d-1)}}\}\quad\text{a.s.},
\end{eqnarray}
where $\|\Sigma_{F_1}\|$ is the operator norm. The proof of this result can be found in \cite[Remark 14]{CK-HL}.

Watanabe \cite{W} has introduced the fractional order Sobolev spaces on the Wiener space and studied their applications to the solutions of stochastic differential equations. An equivalent approach to these spaces using $K$-method  then appeared in \cite{AMR} and \cite{HR}. Now let us recall results in this respect that will be needed in the sequel.
\begin{definition}
For any $0<\alpha<1$ and any $1<p<\infty$, we define
\begin{eqnarray*}
{\mathcal E}_{\alpha}^p=({\mathbb L}^p, {\mathbb D}_1^p)_{\alpha,p},
\end{eqnarray*}
where $(\cdot,\cdot)$ denotes the real interpolation space as in \cite{T}.
\end{definition}

There are several equivalent norms in ${\mathcal E}_{\alpha}^p$ (cf. \cite{T}). The one we shall use is given by Peetre's $K$-method:
\begin{eqnarray*}
\|F\|_{{\mathcal E}^p_{\alpha}}:=\big[\int_0^1[\epsilon^{-\alpha}K(\epsilon,p,F)]^p\frac{d\epsilon}{\epsilon}\big]^{\frac{1}{p}},
\end{eqnarray*}
where
\begin{eqnarray*}
K(\epsilon,p,F):=\inf\{\|F_1\|_p+\epsilon\|F_2\|_{1,p},F_1+F_2=F,F_1,F_2\in{\mathbb L}^p\}.
\end{eqnarray*}

\begin{remark}\label{discrete}
Let $0<\alpha<1$, $1<p<\infty$, then $F\in{\mathcal E}_{\alpha}^p$ if and only if
\begin{eqnarray*}
\sum_{n=1}^\infty 2^{np\alpha}K(2^{-n},p,F)^p<\infty.
\end{eqnarray*}
\end{remark}

For the relationship between ${\mathbb D}_{\alpha}^p$ and ${\mathcal E}_{\alpha}^p$, the following theorem is proved in \cite[Theorem 1.1]{W} directly using specific properties of Wiener functionals.
\begin{theorem}\label{two types}
For every $0<\alpha<1$, $1<p<\infty$ and $\varepsilon>0$, we have
\begin{eqnarray*}
{\mathcal E}_{\alpha+\varepsilon}^p\subset{\mathbb D}_{\alpha}^p\subset{\mathcal E}_{\alpha-\varepsilon}^p.
\end{eqnarray*}
\end{theorem}

By the above theorem we deduce immediately that for every $\alpha>0$, ${\mathbb D}_{\alpha-}^{p}:=\bigcap_{0<\beta<\alpha}{\mathbb D}_{\beta}^{p}=\bigcap_{0<\beta<\alpha}{\mathcal E}_{\beta}^{p}=:{\mathcal E}_{\alpha-}^{p}$ and ${\mathbb D}_{\alpha-}^{\infty-}:=\bigcap_{0<\beta<\alpha}{\mathbb D}_{\beta}^{\infty-}=\bigcap_{0<\beta<\alpha}{\mathcal E}_{\beta}^{\infty-}=:{\mathcal E}_{\alpha-}^{\infty-}$.

\section{Properties of the fractional order Sobolev spaces under local assumptions}
Now we will study local properties of the space ${\mathcal E}_{\alpha}^p$ under local assumptions.
In what follows we denote by $C$ a generic constant which can be different from one formula to another.

Let $F\in{\mathbb D}_{\alpha}^{\infty-}$ such that $F>0$ a.s. and
$1/F\in{\mathbb L}^{\infty-}$. It is well known that if
$\alpha$ is a positive integer, then $1/F\in{\mathbb D}_{\alpha}^{\infty-}$. When $\alpha>0$ is not integer, it is
proved in \cite{W} that $1/F$ only belongs to ${\mathbb D}_{\alpha-}^{\infty-}$.
Our first result is a local version of this result whose proof is based on the $K$-method.
\begin{theorem}\label{inverse}
Let $\alpha=k+\sigma$, $k\in{\mathbb N}$, $0<\sigma\leqslant 1$. Suppose that
$F\in{\mathcal E}_{\alpha}^{\infty-}$ or ${\mathbb D}_{\alpha}^{\infty-}$ is a nonnegative Wiener functional, $G\in {\mathbb D}_{k+1}^{\infty-}$,
and $A$ is a measurable set  such that $G1_A=0$ and
\begin{eqnarray*}
E[|\frac{1}{F}|^p1_{A^c}]<\infty,\quad\text{for every}\quad p> 1.
\end{eqnarray*}
Then we have
\begin{eqnarray*}
\frac{1}{F}\cdot G\in{\mathbb D}_{\alpha-}^{\infty-}.
\end{eqnarray*}
Furthermore, if $\alpha$ is a positive integer, $F,G\in {\mathbb D}^{\infty-}_{\alpha}$, then we have
\begin{eqnarray*}
\frac{1}{F}\cdot G\in{\mathbb D}_{\alpha}^{\infty-}.
\end{eqnarray*}
\end{theorem}

\begin{remark}\label{composite}
In fact, as can be seen from the proof of Theorem \ref{inverse}, for every $\Psi\in C_b^{\infty}$,
we have $\Psi(1/F)\cdot G\in{\mathbb D}_{\alpha-}^{\infty-}$, and
it's also easy to obtain that for any $\beta<\alpha$, $p<p'$, we have
\begin{eqnarray}
\|\Psi(\frac{1}{F})\cdot G\|_{\beta,p}\leqslant C\|G\|_{k+1,p'}. \label{estimate0}
\end{eqnarray}
Furthermore, if there exists $G_1\in{\mathbb D}_{k+1}^{\infty-}$ such that
$G_1=1$ on the set $\{G\neq 0\}$ and $G_11_A=0$, then we have $G=G\cdot G_1$ and by $\Psi(1/F)\cdot G_1\in{\mathbb D}_{\alpha-}^{\infty-}$,
\begin{eqnarray}
\|\Psi(\frac{1}{F})\cdot G\|_{\beta,p}=\|\Psi(\frac{1}{F})\cdot G_1\cdot G\|_{\beta,p}\leqslant C\|G\|_{\beta,p'}.
\label{estimate00}
\end{eqnarray}
\end{remark}

In order to prove Theorem \ref{inverse}, we need the following lemma.
\begin{lemma}\label{important lemma}
Let $0<\alpha<1$, $4<p<\infty$. Suppose that $F\in{\mathbb L}^p$ is a nonnegative Wiener functional and $F_n\in{\mathbb L}^p$
such that
\begin{itemize}
  \item[(i)] \begin{eqnarray*}
        \|F_n-F\|_p\leqslant C2^{-n\alpha};
        \end{eqnarray*}
  \item[(ii)]
        \begin{eqnarray*}
        E[|\frac{1}{F}|^p1_{A^c}]<\infty.
        \end{eqnarray*}
\end{itemize}
Set
\begin{eqnarray*}
\widetilde{F_n}:=F_n\vee 0+2^{-n\alpha}.
\end{eqnarray*}
Then we have
\begin{eqnarray}\label{result}
\sup_n E[|\frac{1}{\widetilde{F_n}}|^{p/4}1_{A^c}]<\infty.
\end{eqnarray}
\end{lemma}

\begin{proof}
Obviously
\begin{eqnarray}\label{close}
\|\widetilde{F_n}-F\|_p\leqslant C2^{-n\alpha},
\end{eqnarray}
\begin{eqnarray}\label{index}
E[|\frac{1}{\widetilde{F_n}}|^q|]\leqslant 2^{nq\alpha}, \quad\forall q>1.
\end{eqnarray}
We define
\begin{eqnarray*}
r_n=|\frac{\widetilde{F_n}-F}{F}|1_{A^c}.
\end{eqnarray*}
Therefore we split (\ref{result}) into two parts:
\begin{eqnarray*}
E[|\frac{1}{\widetilde{F_n}}|^{p/4}1_{A^c}]&=&E[|\frac{1}{\widetilde{F_n}}|^{p/4}1_{A^c}
1_{\{r_n\leqslant\frac{1}{2}\}}]+E[|\frac{1}{\widetilde{F_n}}|^{p/4}1_{A^c}1_{\{r_n>\frac{1}{2}\}}]\\
&=:&\Xi_1+\Xi_2.
\end{eqnarray*}
Since on $A^c$,
\begin{eqnarray*}
\frac{2}{3F}1_{\{r_n\leqslant\frac{1}{2}\}}\leqslant\frac{1}{\widetilde{F_n}}1_{\{r_n\leqslant\frac{1}{2}\}}
\leqslant\frac{2}{F}1_{\{r_n\leqslant\frac{1}{2}\}},
\end{eqnarray*}
by Assumption (ii) we have
\begin{eqnarray*}
\sup_n\Xi_1\leqslant E[|\frac{2}{F}|^{p/4}1_{A^c}1_{\{r_n\leqslant\frac{1}{2}\}}]<\infty.
\end{eqnarray*}
Next let us deal with the term $\Xi_2$.
By (\ref{close}) and Assumption (ii), we have
\begin{eqnarray*}
P(A^c\{r_n>\frac{1}{2}\})&\leqslant& 2^{p/2}E[(\frac{\widetilde{F_n}-F}{F})^{p/2}1_{A^c}]\\
&\leqslant& 2^{p/2}\|\widetilde{F_n}-F\|_p^{p/2}\cdot\|\frac{1}{F}1_{A^c}\|_p^{p/2}\\
&\leqslant&C2^{-np\alpha /2}.
\end{eqnarray*}
Consequently, by (\ref{index}) we have
\begin{eqnarray*}
\sup_n\Xi_2&\leqslant& E[|\frac{1}{\widetilde{F_n}}|^p]^{1/4}\cdot P(A^c\{r_n>\frac{1}{2}\})^{1/2}\\
&\leqslant& \sup_nC2^{np\alpha /4}\cdot 2^{-np\alpha /4}<\infty,
\end{eqnarray*}
and hence we complete the proof.
\end{proof}

\begin{proof}[Proof of Theorem \ref{inverse}]
Since there is no essential difference, we assume $0<\alpha<1$ and $F\in{\mathcal E}_{\alpha}^{\infty-}$ for simplicity.
For every $p'>1$, let $p=18p'$, $r=2p'=p/9$ and $r'=p/4$.
By Remark \ref{discrete}, for each $n$, we can find $F_n\in{\mathbb D}_1^p$ such that
\begin{eqnarray*}
\|F_n-F\|_p\leqslant C2^{-n\alpha},\quad\|F_n\|_{1,p}\leqslant C2^{n(1-\alpha)}.
\end{eqnarray*}
Set
\begin{eqnarray*}
\widetilde{F_n}:=F_n\vee 0+2^{-n\alpha}.
\end{eqnarray*}
By the previous lemma, we have
\begin{eqnarray*}
\sup_n\|\frac{1}{\widetilde{F_n}}1_{A^c}\|_{r'}<\infty.
\end{eqnarray*}
Hence by H\"older inequality, we have
\begin{eqnarray*}
\|D(\frac{1}{\widetilde{F_n}})1_{A^c}\|_{r}\leqslant C\|D\widetilde{F_n}\|_p\cdot\|\frac{1}{\widetilde{F_n}}1_{A^c}\|^2_{r'}\leqslant C2^{n(1-\alpha)}.
\end{eqnarray*}
Therefore, by the Meyer equivalence and $\{DG\neq 0\}\subset \{G\neq 0\}\subset A^c$ we have
\begin{eqnarray*}
\|\frac{1}{\widetilde{F_n}}\cdot G\|_{1,p'}\leqslant C E[(|\frac{1}{\widetilde{F_n}}|^{r}+\|D(\frac{1}
{\widetilde{F_n}})\|_{\mathbb H}^{r})1_{A^c}]^{\frac{1}{r}}\cdot\|G\|_{1,r}
\leqslant C2^{n(1-\alpha)}.
\end{eqnarray*}
A similar calculus gives
\begin{eqnarray*}
\|\frac{1}{\widetilde{F_n}}\cdot G-\frac{1}{F}\cdot G\|_{p'}
\leqslant\|G\|_r\cdot\|F-\widetilde{F_n}\|_p\cdot\|\frac 1 {\widetilde{F_n}}1_{A^c}\|_{r'}\cdot\|\frac 1 {F}1_{A^c}\|_{r'}
\leqslant C2^{-n\alpha}.
\end{eqnarray*}
By Remark \ref{discrete}, we obtain
\begin{eqnarray*}
\|\frac{1}{F}\cdot G\|_{{\mathcal E}_{\beta}^{p'}}<\infty,\quad\forall \beta<\alpha,
\end{eqnarray*}
and this completes the proof.
\end{proof}

If the assumption  $G\in {\mathbb D}_{k+1}^{\infty-}$ in Theorem \ref{inverse} is replaced by $G\in{\mathcal E}_{\alpha}^{\infty-}$,
we have the following theorem.
\begin{theorem}\label{inverse worse}
Let $\alpha=k+\sigma$, $k\in{\mathbb N}$, $0<\sigma\leqslant 1$. Suppose that
$F\in{\mathcal E}_{\alpha}^{\infty-}$ or ${\mathbb D}_{\alpha}^{\infty-}$ is a nonnegative Wiener functional,
$G\in {\mathcal E}_{\alpha}^{\infty-}$ or ${\mathbb D}_{\alpha-}^{\infty-}$ , and $A$ is a measurable set  such that $G1_A=0$ and
\begin{eqnarray*}
E[|\frac{1}{F}|^p1_{A^c}]<\infty,\quad\text{for every}\quad p> 1.
\end{eqnarray*}
Then we have
\begin{eqnarray*}
\frac{1}{F}\cdot G\in{\mathbb D}_{(k+\frac \sigma {2^{k+1}})-}^{\infty-}.
\end{eqnarray*}
\end{theorem}

\begin{proof}
Without loss of generality, we assume $0<\alpha<1$, $F\in{\mathcal E}_{\alpha}^{\infty-}$ and $G\in{\mathcal E}_{\alpha}^{\infty-}$
for simplicity. For every $p'>1$, let
$p,r,r'$ as in the proof of Theorem \ref{inverse}.
By Remark \ref{discrete}, for each $n$, we can find $F_n\in{\mathbb D}_1^p$, $G_n\in{\mathbb D}_1^r$ such that
\begin{eqnarray*}
\|F_n-F\|_p\leqslant C2^{-n\alpha},\quad\|F_n\|_{1,p}\leqslant C2^{n(1-\alpha)}
\end{eqnarray*}
and
\begin{eqnarray*}
\|G_n-G\|_r\leqslant C2^{-n\alpha},\quad\|G_n\|_{1,r}\leqslant C2^{n(1-\alpha)}.
\end{eqnarray*}
Let
\begin{eqnarray*}
\widetilde{F_n}:=F_n\vee 0+2^{-n\alpha/2}.
\end{eqnarray*}
Then $\widetilde{F_n}$ satisfies
\begin{eqnarray*}
\|\widetilde{F_n}-F\|_p\leqslant C2^{-n\alpha/2},\quad\|\widetilde{F_n}\|_{1,p}\leqslant C2^{n(1-\alpha/2)}.
\end{eqnarray*}
Therefore we have
\begin{eqnarray*}
\|\frac{1}{\widetilde{F_n}}\cdot G_n-\frac{1}{F}\cdot G\|_{p'}&\leqslant&
\|\frac{1}{\widetilde{F_n}}\cdot G-\frac{1}{F}\cdot G\|_{p'}+\|\frac 1 {\widetilde{F_n}}(G_n-G)\|_{p'}\\
&=:&\Lambda_1+\Lambda_2.
\end{eqnarray*}
By the same method as in the proof of Theorem \ref{inverse}, we have
\begin{eqnarray*}
\Lambda_1\leqslant C2^{-n\alpha/2}.
\end{eqnarray*}
By $1/\widetilde{F_n}\leqslant 2^{n\alpha/2}$ and $\|G_n-G\|_r\leqslant C2^{-n\alpha}$, we also have
\begin{eqnarray*}
\Lambda_2\leqslant\|G_n-G\|_r\cdot\|\frac{1}{\widetilde{F_n}}\|_{r}\leqslant C2^{-n\alpha/2}.
\end{eqnarray*}
Consequently we have
\begin{eqnarray*}
\|\frac{1}{\widetilde{F_n}}\cdot G_n-\frac{1}{F}\cdot G\|_{p'}\leqslant C2^{-n\alpha/2}.
\end{eqnarray*}
Similarly, by Meyer equivalence of norms,
we have
\begin{eqnarray*}
\|\frac{1}{\widetilde{F_n}}\cdot G_n\|_{1,p'}&\leqslant& C\{\|\frac{1}{\widetilde{F_n}}\cdot G_n\|_{p'}+\|D(\frac{1}{\widetilde{F_n}}\cdot G_n)\|_{p'}\}\\
&=&C\{\|\frac{1}{\widetilde{F_n}}\cdot G_n\|_{p'}+\|\frac{1}{\widetilde{F_n}^2}\cdot D\widetilde{F_n}\cdot G_n\|_{p'}
+\|\frac{1}{\widetilde{F_n}}\cdot DG_n\|_{p'}\}\\
&\leqslant&C\{\|\frac{1}{\widetilde{F_n}}\cdot G_n\|_{p'}+\|\frac{1}{\widetilde{F_n}^2}\cdot D\widetilde{F_n}\cdot (G_n-G)\|_{p'}\\
&&{}+\|\frac{1}{\widetilde{F_n}^2}\cdot D\widetilde{F_n}\cdot G\|_{p'}+\|\frac{1}{\widetilde{F_n}}\cdot DG_n\|_{p'}\}.
\end{eqnarray*}
Similar to the proof of Theorem \ref{inverse}, we have
\begin{eqnarray*}
\|\frac{1}{\widetilde{F_n}^2}\cdot D\widetilde{F_n}\cdot G\|_{p'}\leqslant C2^{n(1-\alpha/2)}.
\end{eqnarray*}
By $1/\widetilde{F_n}\leqslant 2^{n\alpha/2}$, $\|G_n-G\|_r\leqslant C2^{-n\alpha}$ and $\|D\widetilde{F_n}\|_p\leqslant C2^{n(1-\alpha/2)}$, we have
\begin{eqnarray*}
\|\frac{1}{\widetilde{F_n}^2}\cdot D\widetilde{F_n}\cdot (G_n-G)\|_{p'}&\leqslant&
\|\frac{1}{\widetilde{F_n}}\|_{r'}^2\cdot\|D\widetilde{F_n}\|_p\cdot\|G_n-G\|_r\\
&\leqslant& C2^{n(1-\alpha/2)}.
\end{eqnarray*}
By $1/\widetilde{F_n}\leqslant 2^{n\alpha/2}$, $\|G_n-G\|_r\leqslant C2^{-n\alpha}$ and $\|G_n\|_{1,r}\leqslant 2^{n(1-\alpha)}$, we also have
\begin{eqnarray*}
\|\frac{1}{\widetilde{F_n}}\cdot G_n\|_{p'}\leqslant \|\frac{1}{\widetilde{F_n}}\cdot (G_n-G)\|_{p'}+\|\frac{1}{\widetilde{F_n}}\cdot G\|_{p'}
\leqslant C,
\quad\|\frac{1}{\widetilde{F_n}}\cdot DG_n\|_{p'}\leqslant C2^{n(1-\alpha/2)}.
\end{eqnarray*}
Consequently, we have
\begin{eqnarray*}
\|\frac{1}{\widetilde{F_n}}\cdot G_n\|_{1,p'}\leqslant C2^{n(1-\alpha/2)}.
\end{eqnarray*}
By Remark \ref{discrete}, we obtain
\begin{eqnarray*}
\|\frac{1}{F}\cdot G\|_{{\mathcal E}_{\beta}^{p'}}<\infty,\quad\forall \beta<\alpha/2,
\end{eqnarray*}
and this completes the proof.
\end{proof}

The following theorem is a local version of Lemma 2.1 in \cite{W} which will play an essential role in Section 4.
Before proceeding, we introduce some notions and notations concerning the Sobolev space on ${\mathbb R}^d$
(cf. \cite{A} and \cite{T}). Define the family of Bessel potential spaces by
\begin{eqnarray*}
{\mathbb L}_{\alpha}^p({\mathbb R}^d)=(1-\Delta)^{-\alpha/2}({\mathbb L}^p({\mathbb R}^d)),\quad 1<p<\infty,\quad\alpha\in {\mathbb R}
\end{eqnarray*}
equipped with the norm
\begin{eqnarray*}
\|g\|_{\alpha,p}=\|(1-\Delta)^{\alpha/2}g\|_p.
\end{eqnarray*}
 For $1<p<\infty$, $\alpha\in{\mathbb R}$, we also define fractional order Sobolev spaces on ${\mathbb R}^d$ as follows:
\begin{eqnarray*}
{\mathcal E}_{\alpha}^{p}({\mathbb R}^d)=({\mathbb L}_{k}^p({\mathbb R}^d),{\mathbb L}_{k+1}^p({\mathbb R}^d))_{\sigma,p}
\quad\text{if $k+1$ be the smallest integer larger than $\alpha$}
\end{eqnarray*}
equipped with the norm
\begin{eqnarray*}
\|g\|_{{\mathcal E}_{\alpha}^{p}({\mathbb R}^d)}=[\int_0^{1}|\epsilon^{-\sigma}K(\epsilon,p,g)|^p\frac{d\epsilon}{\epsilon}]^{1/p},
\end{eqnarray*}
where $\sigma:=\alpha-k$ and
\begin{eqnarray*}
K(\epsilon,p,g):=\inf\{\|g_1\|_{k,p}+\epsilon\|g_2\|_{k+1,p},g_1+g_2=g,g_1,g_2\in{\mathbb L}_k^p({\mathbb R}^d)\}.
\end{eqnarray*}

\begin{theorem}\label{local version1}
Suppose that $F\in{\mathbb D}_1^{\infty-}({\mathbb R}^d)$, $G\in{\mathbb D}_1^{\infty-}$ taking values on $[0,1]$ with
\begin{eqnarray*}
1+E_G[\|D\ln G\|_{\mathbb H}^p]<\infty,\quad\text{for every}\quad p\geq 1,
\end{eqnarray*}
and the density $p_{F,G}$ is bounded. Then, for every $1<p<p'<\infty$ and $0\leqslant\alpha<\alpha'\leqslant1$,
there exists a positive constant $C=C(\alpha,\alpha',p,p')$ such that
\begin{eqnarray*}
\|g\circ F\cdot G\|_{{\mathcal E}_{\alpha}^p}\leqslant C\|g\|_{{\mathcal E}_{\alpha'}^{p'}({\mathbb R}^d)}\quad\text{for every}\quad g\in S({\mathbb R}^d).
\end{eqnarray*}
\end{theorem}

\begin{proof}
Since $p_{F,G}$ is bounded, for every $p'>p$, we can find $C=C(p,p')>0$ such that
\begin{eqnarray}\label{1}
\|g\circ F\cdot G\|_p&\leqslant&\|g\circ F\cdot G^{1/p'}\|_{p'}\cdot\|G^{1-1/p'}\|_{pp'/(p'-p)}\nonumber\\
&=&[\int_{{\mathbb R}^d}|g(x)|^{p'}p_{F,G}(x)dx]^{1/p'}\cdot\|G^{1-1/p'}\|_{pp'/(p'-p)}\nonumber\\
&\leqslant& C\|g\|_{p'}.
\end{eqnarray}
Also, by the Meyer equivalence of norms, for every $p'>p$, we have
\begin{eqnarray*}
\|g\circ F\cdot G\|_{1,p}&=&\|(1-L)^{1/2}(g\circ F\cdot G)\|_p\\
&\leqslant&C\{\|g\circ F\cdot G\|_p+\|D(g\circ F\cdot G)\|_p\}\\
&\leqslant&C\{\|g\circ F\cdot G\|_p+\|\sum_{i=1}^d(\partial_ig)\circ F\cdot DF_i\cdot G\|_p+
\|g\circ F\cdot DG\|_p\}\\
&\leqslant&C\{\|g\circ F\cdot G\|_p+\|\sum_{i=1}^d(\partial_ig)\circ F\cdot DF_i\cdot G\|_p+
\|g\circ F\cdot D(\ln G)\cdot G\|_p\}\\
&\leqslant&C\{\|g\circ F\cdot G\|_p+\sum_{i=1}^d\|(\partial_ig)\circ F\cdot G\|_{p'}\cdot\|DF_i\|_{pp'/(p'-p)}\\
&&+\|g\circ F\cdot G^{1-((p'-p)/pp')}\|_{p'}\cdot\|D(\ln G)\cdot G^{(p'-p)/pp'}\|_{pp'/(p'-p)}\}.
\end{eqnarray*}
Then similarly as (\ref{1}), for every $p'>p$, we can find $C=C(p,p')$ such that
\begin{eqnarray}\label{2}
\|g\circ F\cdot G\|_{1,p}\leqslant C\|g\|_{1,p'}.
\end{eqnarray}
By (\ref{1}) and (\ref{2}), we get the desired estimate if $\alpha=0$ or $\alpha=1$. If $0<\alpha<1$, we
proceed as follows. Let $1<p<p'<\infty$ and $0<\alpha<\alpha'<1$ be fixed. For $g\in{\mathcal E}_{\alpha'}^{p'}({\mathbb R}^d)$,
by (\ref{1}) and (\ref{2}), we have
\begin{eqnarray*}
K(\epsilon,p,g\circ F\cdot G)\leqslant CK(\epsilon,p',g).
\end{eqnarray*}
Hence by H\"older inequality we obtain that
\begin{eqnarray*}
\|g\circ F\cdot G\|_{{\mathcal E}_{\alpha}^p}&=&\big[ \int_0^1 [\epsilon^{-\alpha} K(\epsilon,p,g\circ F\cdot G)]^p
\frac{d\epsilon}{\epsilon} \big]^{\frac{1}{p}}\\
&\leqslant &C\big[ \int_0^1 [\epsilon^{-\alpha} K(\epsilon,p',g)]^p
\frac{d\epsilon}{\epsilon} \big]^{\frac{1}{p}}\\
&= &C\big[ \int_0^1 [\epsilon^{-\alpha'p-\frac p {p'}} K(\epsilon,p',g)^p][\epsilon^{(\alpha'-\alpha)p-\frac {p'-p} {p'}}]
{d\epsilon} \big]^{\frac{1}{p}}\\
&\leqslant &C\big[ \int_0^1 [\epsilon^{-\alpha'} K(\epsilon,p',g)]^{p'}
\frac{d\epsilon}{\epsilon} \big]^{\frac{1}{p'}}\big[ \int_0^1 \epsilon^{\frac {(\alpha'-\alpha)p'p} {p'-p}}
\frac{d\epsilon}{\epsilon} \big]^{\frac{p'-p}{p'p}}\\
&\leqslant &C\|g\|_{{\mathcal E}_{\alpha'}^{p'}({\mathbb R}^d)}
\end{eqnarray*}
as desired.
\end{proof}

We shall need the following local version of the integration by parts formula.
\begin{lemma}\label{integration by parts}
Let $k\in{\mathbb N}$. Suppose
\begin{itemize}
  \item[(i)] $F\in{\mathbb D}_{k+2}^{\infty-}({\mathbb R}^d)$;
  \item[(ii)] $G\in{\mathbb D}_{k+1}^{\infty-}$, $A$ is a measurable set such that $G1_A=0$
      and for any $p> 1$,
      \begin{eqnarray*}
      E[|\det(\Sigma_F)|^{-p}1_{A^c}]<\infty.
      \end{eqnarray*}
\end{itemize}
Then for all $g\in C_p^{\infty}({\mathbb R}^d)$ and all multiindex $\alpha\in\{1,\cdots,d\}^{k+1}$, we have
\begin{eqnarray}
E[(\frac{\partial^{|\alpha|} g}{\partial y^{\alpha}})\circ F\cdot G]=E[g\circ F\cdot H_{\alpha}(F,G)]
\end{eqnarray}
where the elements $H_{\alpha}(F,G)\in {\mathbb D}^{\infty-}_{k+1-|\alpha|}$ are recursively given by
\begin{eqnarray*}
H_{(i)}(F,G)&=&H_i(F,G)\\
&:=&\sum_{j=1}^dD^*(\gamma_F^{ij}\cdot G\cdot DF_j)\\
&=&-\sum_{j=1}^d\{\gamma_F^{ij}\cdot LF_j\cdot G+(D\gamma_F^{ij},DF_j)_{\mathbb H}\cdot G+\gamma_F^{ij}\cdot(DG,DF_j)_{\mathbb H}\},\\
H_{\alpha}(F,G)&=&H_{(i_1,\cdots,i_k)}(F,G)\\
&:=&H_{i_k}(F,H_{(i_1,\cdots,i_{k-1})}(F,G)),
\end{eqnarray*}
which satisfy $H_{\alpha}(F,G)1_A=0$ and
\begin{eqnarray}
\|H_\alpha(F,G)\|_{k+1-|\alpha|,p}\leqslant C\|G\|_{k+1,p'}
\end{eqnarray}
for every $p<p'$.
\end{lemma}

\begin{proof}
Since
\begin{eqnarray*}
(\det(\Sigma_F))^{-1}1_{A^c}\in {\mathbb L}^{\infty-} \quad \text{and}\quad \det(\Sigma_F)\in{\mathbb D}_{k+1}^{\infty-},
\end{eqnarray*}
by Theorem \ref{inverse}, we have
\begin{eqnarray*}
(\det(\Sigma_F))^{-1}G\in {\mathbb D}^{\infty-}_{k+1},
\end{eqnarray*}
which in turn yields that
\begin{eqnarray*}
\gamma_F^{ij}\cdot G\in {\mathbb D}^{\infty-}_{k+1}.
\end{eqnarray*}
By the chain rule we have
\begin{eqnarray*}
(\partial_i g)\circ F \cdot 1_{A^c}=\sum_{j=1}^d(D (g\circ F),DF_j)_{\mathbb H}\cdot \gamma_F^{ij}\cdot 1_{A^c},
\end{eqnarray*}
so we obtain
\begin{eqnarray*}
(\partial_i g)\circ F\cdot G=(D (g\circ F),\sum_{j=1}^d \gamma_F^{ij}\cdot G\cdot DF_j)_{\mathbb H}.
\end{eqnarray*}
Hence by the duality relationship between the derivative and the divergence operators, we get
\begin{eqnarray*}
E[(\partial_i g)\circ F\cdot G]=E[g\circ F\cdot H_i(F,G)],
\end{eqnarray*}
where
\begin{eqnarray*}
H_i(F,G)&=&\sum_{j=1}^dD^*(\gamma_F^{ij}\cdot G\cdot DF_j)\\
&=&-\sum_{j=1}^d\{\gamma_F^{ij}\cdot LF_j\cdot G+(D\gamma_F^{ij},DF_j)_{\mathbb H}\cdot G+
\gamma_F^{ij}\cdot(DG,DF_j)_{\mathbb H}\}.
\end{eqnarray*}
It follows by the localness of $D$ that $H_i(F,G)1_A=0$, and by \eqref{estimate0} and
the fact that  $D^*:{\mathbb D}_{1+\alpha}^p({\mathbb H})\rightarrow{\mathbb D}_{\alpha}^p$
is continuous for every $p$ and every $\alpha$, we see that for every $p<p'$,
\begin{eqnarray*}
\|H_i(F,G)\|_{k,p}\leqslant C\sum_{j=1}^d\|\gamma_F^{ij}\cdot G\cdot DF_j\|_{k+1,p}\leqslant C\|G\|_{k+1,p'}.
\end{eqnarray*}
Thus we have
\begin{eqnarray*}
E[(\partial_{i_1}\partial_{i_2}g)\circ F\cdot G]=E[g\circ F\cdot H_{i_2}(F,H_{i_1}(F,G))].
\end{eqnarray*}
Obviously we still have
$H_{i_2}(F,H_{i_1}(F,G))1_A=0$ and consequently we can go ahead further.
Moreover, it is easy to see that for every $p<p'<p''$,
\begin{eqnarray*}
\|H_{i_2}(F,H_{i_1}(F,G))\|_{k-1,p}\leqslant C\|H_{i_1}(F,G)\|_{k,p'}\leqslant C\|G\|_{k+1,p''}.
\end{eqnarray*}
By induction we  prove the desired results.
\end{proof}

\begin{remark}\label{localness}
In applications one usually takes $A=\{G= 0\}$. Note that by the localness of the derivative
operator one then has $\{D^\beta G\neq 0\}\subset A^c$ for any $\beta\in{\mathbb N}$.
\end{remark}

\begin{remark}\label{estimate4}
Under the conditions of Lemma \ref{integration by parts}, suppose that there exists $G_1\in{\mathbb D}_{k}^{\infty-}$ such that
$G_1=1$ on the set $\{G\neq 0\}$ and $G_11_A=0$, then for any $G'\in \mathbb D^{\infty-}_{k+1}$ and
$\varepsilon\geqslant 0$, we have $H_{\alpha}(F,G\cdot G')=H_{\alpha}(F,G\cdot G')\cdot G_1$,
\begin{eqnarray*}
\|H_i(F,G\cdot G')\|_{k-\varepsilon,p}\leqslant C\sum_{j=1}^d\|\gamma_F^{ij}\cdot G
\cdot G'\cdot DF_j\|_{k+1-\varepsilon,p}\leqslant C\|G'\|_{k+1-\varepsilon,p'}
\quad \text{\ for every\ }p<p'.
\end{eqnarray*}
Since $\gamma_F^{ij}\cdot G_1\in {\mathbb D}_{k}^{\infty-}$,
\begin{eqnarray*}
&&\|H_{(i_1,i_2)}(F,G\cdot G')\|_{k-1-\varepsilon,p}\\
&\leqslant& C\sum_{j=1}^d\|\gamma_F^{i_2j}\cdot H_{i_1}(F,G\cdot G')\cdot G_1\cdot DF_j\|_{k-\varepsilon,p}\\
&\leqslant& C\|H_{i_1}(F,G\cdot G')\|_{k-\varepsilon,p'}\leqslant C\|G'\|_{k+1-\varepsilon,p''}\quad \text{\ for every\ }p<p'<p''.
\end{eqnarray*}
Therefore by induction we have
\begin{eqnarray}
\|H_{\alpha}(F,G\cdot G')\|_{k+1-|\alpha|-\varepsilon,p}\leqslant C\|G'\|_{k+1-\varepsilon,p'}
\end{eqnarray}
for every $p<p'$.
\end{remark}

\section{Estimate of the difference between two Donsker's delta functions}
In this section, we establish an estimate of the difference between two Donsker's delta functionals. In what follows we also denote by $C$ a generic constant which can be different from one formula to another.
\begin{theorem}\label{comparison theorem}
Let $\delta=k+\sigma$, $k\in{\mathbb N}$, $0<\sigma\leqslant1$. Suppose that $H,H_1,H_2:B\rightarrow[0,1]$ and $F_1,F_2:B\rightarrow {\mathbb R}^d$ are
Wiener functionals such that
\begin{itemize}
  \item[(i)] $F_1,F_2\in{\mathbb D}_{2+\delta}^{\infty-}({\mathbb R}^d)$;
  \item[(ii)] $H,H_1,H_2\in{\mathbb D}_{k+2}^{\infty-}$ with
  \begin{eqnarray*}
  1+E_{H_2}[\|D\ln H_2\|_{\mathbb H}^p]<\infty\quad\text{for every}\quad p\geq1,
  \end{eqnarray*}
  such that $H_1=1$ on the set $\{H\neq0\}$ and $H_2=1$ on the set $\{H_1\neq0\}$;
  \item[(iii)] there is a measurable set $A$ such that $H_21_A=0$ and $F_1,F_2$ are nondegenerate a.s. on the set $A^c$.
\end{itemize}
Then for every $p$, $p'$, $p''$, $r_1$, $r_2$ and $r_3$ satisfying $1<p<p'<p''<\infty$, $r_1>pp'/(p'-p)$ and $r_2>r_3=p'p''/(p''-p')$, and for
\begin{eqnarray*}
0<\beta<\alpha\wedge(1+\delta)-1-\frac{d(p-1)}{p},\quad \beta+\frac{d(p''-1)}{p''}<\delta'<\delta,\quad \delta'\leqslant \alpha-1,
\end{eqnarray*}
we can find a positive constant $C$ which may depend on $F_1$, $F_2$, $\alpha$, $\beta$, $\delta$, $\delta'$, $p$, $p'$, $p''$, $r_1$ and $r_2$
such that
\begin{eqnarray*}
&&\|(1-\Delta)^{\beta/2}\delta_x\circ F_1\cdot H-(1-\Delta)^{\beta/2}\delta_x\circ F_2\cdot H\|_{-\alpha,p}\nonumber\\
&\leqslant& C\|F_1-F_2\|_{2+\delta',r_1}+C\|F_1-F_2\|_{1,r_2}+C\|F_1-F_2\|_{\delta',r_3}.
\end{eqnarray*}
\end{theorem}

\begin{remark}
The expression $(1-\Delta)^{\beta/2}\delta_x\circ F$ is well-defined. We refer the reader to \cite[Definition 2.1]{W} for more details.
\end{remark}

\begin{remark}
Since $p_{F,G}(x)=E[1_{\{F>x\}}H_{(1,\cdots,d)}(F,G)]$ (see \cite[Proposition 2.1.5]{N}), by \eqref{important result1}, if $k \geqslant d-1$ then
the conditions (iii), $F_1,F_2\in{\mathbb D}_{k+2}^{\infty-}({\mathbb R}^d)$, $H_2\in {\mathbb D}_{k+1}^{\infty-}$ and $H_21_A=0$ imply that
\begin{enumerate}\label{bally remark}
  \item for $i=1,2$, the densities $p_{F_i,H_2}$ of the laws of $F_i$ under $\mu_{H_2}$ are bounded;
  \item for every $1<p<\infty$, the density $p_{F_2+t(F_1-F_2),H_2\cdot\Psi_1(R_{F_1,F_2})}$ is bounded, uniformly in $t\in[0,1]$.
\end{enumerate}
Recently, Bally and Caramellino \cite{BC1} have proved that any non-degenerated functional which is twice differentiable in
Malliavin sense has a bounded  density (see Proposition \ref{bally}). Using this result, the conditions (1) and (2) are
automatically satisfied under the conditions of Theorem \ref{comparison theorem} (see \cite[Example 2.3]{BC2} for $\Psi_1$
satisfying the condition \eqref{bally example}).
This fact will play a very important role in the proof of Theorem \ref{comparison theorem}.
\end{remark}

Now we state the following corollary which is more refined than
\cite[Proposition 2.2]{BC2}.
\begin{corollary}
In the circumstance of Theorem \ref{comparison theorem} we have
\begin{eqnarray*}
\|p_{F_1,H}-p_{F_2,H}\|_{C^\beta({\mathbb R}^d)}\leqslant C\|F_1-F_2\|_{2+\delta',r_1\vee r_2},
\end{eqnarray*}
\end{corollary}

\begin{proof}
Since $p_{F_i,H}(x)=E_H[\textbf{1}\cdot\delta_x\circ F_i]$, we have
\begin{eqnarray*}
\|(1-\Delta)^{\beta/2}(p_{F_1,H}(x)-p_{F_2,H}(x))\|_{C^0({\mathbb R}^d)}\leqslant C\|F_1-F_2\|_{2+\delta',r_1\vee r_2}.
\end{eqnarray*}
Then the conclusion follows from the fact that $(1-\Delta)^{-{\beta}/{2}}(C^0({\mathbb R}^d))\subset C^{\beta}({\mathbb R}^d)$.
\end{proof}

To prove the theorem we need some preparations.
In what follows for the simplicity of notations sometimes we shall drop the notation of $\sum$ if
there is no confusion. We begin with the following lemma which can be found in Watanabe \cite{W}.
\begin{lemma}\label{w1}
If $\alpha>d(p-1)/p$, then the map $x\in{\mathbb R}^d\rightarrow(1-\Delta)^{-\alpha/2}
\delta_x\in \mathbb{L}^p({\mathbb R}^d)$ is bounded and continuous.
\end{lemma}

The following lemma is a local version of Lemma 2.2 in \cite{W}.
\begin{lemma}\label{local version2}
Let $\delta=k+\sigma$, $k\in{\mathbb N}$, $0<\sigma\leqslant 1$. Suppose that $G:B\rightarrow {\mathbb R}$ and
$F:B\rightarrow {\mathbb R}^d$ be Wiener functionals such that
\begin{itemize}
  \item[(i)]$F\in{\mathbb D}_{2+\delta}^{\infty-}({\mathbb R}^d)$;
  \item[(ii)]$G\in{\mathbb D}_{k+1}^{\infty-}$, $G_1\in{\mathbb D}_{k+1}^{\infty-}$ taking values on $[0,1]$ with
  \begin{eqnarray*}
  1+E_{G_1}[\|D\ln G_1\|_{\mathbb H}^p]<\infty\quad\text{for every}\quad p\geq1,
  \end{eqnarray*}
  such that $G_1=1$ on the set $\{G\neq 0\}$;
  \item[(iii)]there is a measurable set $A$ such that $G_11_A=0$ and $F$ is nondegenerate a.s. on the set $A^c$.
\end{itemize}
Then, for every $1<p<p'<\infty$ and $0<\delta'<\widetilde{\delta}\leqslant\delta$,
we can find a positive constant $C=C(\tilde{\delta},\delta',p,p')$ such that
\begin{eqnarray*}
\|g\circ F\cdot G\|_{-\widetilde{\delta},p}\leqslant C\|g\|_{-\delta',p'}\quad\text{for every}\quad g\in S({\mathbb R}^d).
\end{eqnarray*}
\end{lemma}

\begin{proof}
By Lemma \ref{integration by parts}, we obtain for $G'\in{\mathbb D}_{k+1}^{\infty-}$,
\begin{eqnarray*}
E[(\partial_{i_1}\cdots\partial_{i_m}g)\circ F\cdot G\cdot G']=E[g\circ F\cdot H_{(i_1,\cdots, i_m)}(F,G\cdot G')].
\end{eqnarray*}
Taking $\varepsilon=k+1-\widetilde{\delta}$, by Remark \ref{estimate4}, we can find, for every $1<q'<q$,
$\widetilde{\delta}>0$ and $m=0,1,\cdots,k+1$, a positive constant
$C=C(q,q',\widetilde{\delta},m;F,G,G_1,G')$ such that
\begin{eqnarray*}
\|H_{(i_1,\cdots, i_m)}(F,G\cdot G')\|_{\widetilde{\delta}-m,q'}\leqslant C\|G'\|_{\widetilde{\delta},q}.
\end{eqnarray*}
Consequently, if $1/p+1/q=1$ and $1/p'+1/q'=1$, we have
\begin{eqnarray}\label{intermediate results1}
&&\|(\partial_{i_1}\cdots\partial_{i_m}g)\circ F\cdot G\|_{-\widetilde{\delta},p}\nonumber\\
&=&\sup\{|E[(\partial_{i_1}\cdots\partial_{i_m}g)\circ F\cdot G\cdot G']|:\|G'\|_{\widetilde{\delta},q}\leqslant 1\}\nonumber\\
&\leqslant&\sup\{|E[g\circ F\cdot H_{(i_1,\cdots, i_m)}(F,G\cdot G')]|:\|H_{(i_1,\cdots, i_m)}(F,G\cdot G')\|_{\widetilde{\delta}-m,q'}\leqslant C\}\nonumber\\
&=&\sup\{|E[g\circ F\cdot H_{(i_1,\cdots, i_m)}(F,G\cdot G')\cdot G_1]|:\|H_{(i_1,\cdots, i_m)}(F,G\cdot G')\|_{\widetilde{\delta}-m,q'}\leqslant C\}\nonumber\\
&\leqslant& C\|g\circ F\cdot G_1\|_{m-\widetilde{\delta},p'}\quad\text{for}\quad m=0,\cdots,k+1.
\end{eqnarray}
The rest of the proof is the same as that of \cite{W} and we give it for reader's convenience.
If $k=2l-1$ is odd, any $g\in S({\mathbb R}^d)$ can be written as
\begin{eqnarray*}
g=\sum_{n=0}^l\Sigma'\pm\partial_{i_1}\cdots\partial_{i_{2n}}(1-\Delta)^{-l}g,
\end{eqnarray*}
where $\Sigma'$ is a certain sum over indices $(i_1,\cdots,i_{2n})$. Hence, by (\ref{intermediate results1}), we have
\begin{eqnarray*}
&&\|g\circ F\cdot G\|_{-\widetilde{\delta},p}\\
&\leqslant& C\sum_{n=0}^l\|(1-\Delta)^{-l}g\circ F\cdot G_1\|_{2n-\widetilde{\delta},p'}\\
&\leqslant& C\|(1-\Delta)^{-(k+1)/2}g\circ F\cdot G_1\|_{k+1-\widetilde{\delta},p'}.
\end{eqnarray*}
Since
\begin{eqnarray*}
0\leqslant1-\sigma=k+1-\widetilde{\delta}<1,
\end{eqnarray*}
using Remark \ref{bally remark} (since $p_{F,G_1}$ is bounded) and Theorem \ref{local version1}
we have for every $0<\delta'<\widetilde{\delta}$ and $p''>p'$ with $k+1-\delta'<1$,
\begin{eqnarray*}
\|g\circ F\cdot G\|_{-\widetilde{\delta},p}\leqslant C\|(1-\Delta)^{-(k+1)/2}g\|_{k+1-\delta',p''}=C\|g\|_{-\delta',p''}.
\end{eqnarray*}
This yields the desired estimate, since $\delta'$ and $p''$ can be chosen arbitrarily close to $\widetilde{\delta}$ and $p$.
Similar arguments applies for even $k$ and we refer to \cite[Lemma 2.2]{W} for details. The proof is  completed.
\end{proof}

We also need the following technical lemma.
\begin{lemma}\label{key estimate}
Let $H,H_1\in B\rightarrow[0,1]$ and $F_1,F_2\in B\rightarrow{\mathbb R}^d$ satisfy all the conditions (i), (ii) and (iii) of
Theorem \ref{comparison theorem}.
Then, for every $\delta',r,r',r''$ satisfying $0<\delta'<{\delta}$,
$r'>r>1$ and $r''>r'r/(r'-r)$, we can find a
positive constant $C=C(F_1,F_2,\delta',r,r',r'')$ such that
\begin{eqnarray*}
\|H_i(F_1,H\cdot G)-H_i(F_2,H\cdot G)\|_{\delta',r}\leqslant C\|G\|_{1+\delta',r'}\cdot\|F_1-F_2\|_{2+\delta',r''},
\end{eqnarray*}
where $G\in{\mathbb D}_{1+\delta}^{\infty-}$.
\end{lemma}

\begin{proof}
In what follows, $r,m,n,m',m'^c,r',r'^c,r'',r''^c,r''^{cc}$ satisfy
\begin{eqnarray*}
\frac{1}{m}+\frac{1}{n}=\frac{1}{r},\frac{1}{m'}+\frac{1}{m'^c}=\frac{1}{m},
\frac{1}{r'}+\frac{1}{r'^c}=\frac{1}{n},
\frac{1}{r''^c}+\frac{1}{r''^{cc}}=\frac{1}{m'^c},
m'<r''.
\end{eqnarray*}
Since the Malliavin covariance matrix $\Sigma_{F_1}$ satisfies that
$\sigma_{F_1}^{ij}\in {\mathbb D}_{1+\delta}^{\infty-}$,
and the space ${\mathbb D}_{1+\delta}^{\infty-}$ is an algebra, so $\det(\Sigma_{F_1})\in {\mathbb D}_{1+\delta}^{\infty-}$,
then by Theorem \ref{inverse}, we have
\begin{eqnarray*}
(\det(\Sigma_{F_1}))^{-1}\cdot H\in {\mathbb D}_{(1+\delta)-}^{\infty-},
\end{eqnarray*}
and we deduce easily from this that $\gamma_{F_1}^{ij}\cdot H\in {\mathbb D}_{(1+\delta)-}^{\infty-}$. Since
\begin{eqnarray*}
H_i(F_1,H\cdot G)=\sum_{j=1}^dD^*(\gamma_{F_1}^{ij}\cdot H\cdot G\cdot DF_1^j)
\end{eqnarray*}
and
\begin{eqnarray*}
H_i(F_2,H\cdot G)=\sum_{j=1}^dD^*(\gamma_{F_2}^{ij}\cdot H\cdot G\cdot DF_2^j),
\end{eqnarray*}
using the fact that $D^*:{\mathbb D}_{1+\alpha}^p({\mathbb H})\rightarrow{\mathbb D}_{\alpha}^p$ is continuous for every $p$
and $\alpha$, we have
\begin{eqnarray*}
\|H_i(F_1,H\cdot G)-H_i(F_2,H\cdot G)\|_{\delta',r}&\leqslant&C\|\gamma_{F_1}^{ij}\cdot H\cdot G\cdot DF_1^j-
\gamma_{F_2}^{ij}\cdot H\cdot G\cdot DF_2^j\|_{1+\delta',r}\\
&\leqslant&C\|\gamma_{F_1}^{ij}\cdot H\cdot G\cdot DF_1^j-\gamma_{F_2}^{ij}\cdot H\cdot G\cdot DF_1^j\|_{1+\delta',r}\\
&&{}+C\|\gamma_{F_2}^{ij}\cdot H\cdot G\cdot DF_1^j-\gamma_{F_2}^{ij}\cdot H\cdot G\cdot DF_2^j\|_{1+\delta',r}\\
&=:&\Lambda_1+\Lambda_2.
\end{eqnarray*}
We first observe that
\begin{eqnarray*}
(\Sigma_{F_2})^{-1}-(\Sigma_{F_1})^{-1}=(\Sigma_{F_1})^{-1}(\Sigma_{F_1}-\Sigma_{F_2})(\Sigma_{F_2})^{-1}.
\end{eqnarray*}
Therefore we have
\begin{eqnarray*}
\gamma_{F_2}^{ij}-\gamma_{F_1}^{ij}=\sum_{k,l=1}^d\gamma_{F_1}^{ik}(\sigma_{F_1}^{kl}-\sigma_{F_2}^{kl})
\gamma_{F_2}^{lj}.
\end{eqnarray*}
In view of $\gamma_{F_1}^{ij}\cdot H\in{\mathbb D}_{(1+\delta)-}^{\infty-}$, $\gamma_{F_2}^{ij}\cdot H_1\in
{\mathbb D}_{(1+\delta)-}^{\infty-}$, $F_1\in{\mathbb D}_{2+\delta}^{\infty-}({\mathbb R}^d)$,
$F_2\in{\mathbb D}_{2+\delta}^{\infty-}({\mathbb R}^d)$, we obtain
\begin{eqnarray}\label{B1}
\Lambda_1&\leqslant&C\|G\cdot DF_1^j\|_{1+\delta',n}\cdot \|\gamma_{F_2}^{ij}\cdot H-\gamma_{F_1}^{ij}\cdot H\|_{1+\delta',m}\nonumber\\
&\leqslant& C\|G\|_{1+\delta',r'}\cdot\|DF_1^j\|_{1+\delta',r'^c}\cdot\|\sigma_{F_1}^{kl}-
\sigma_{F_2}^{kl}\|_{1+\delta',m'}\cdot\|\gamma_{F_1}^{ik}\cdot H\|_{1+\delta',r''^c}\cdot
\|\gamma_{F_2}^{lj}\cdot H_1\|_{1+\delta',r''^{cc}}\nonumber\\
&\leqslant& C\|G\|_{1+\delta',r'}\cdot\|\sigma_{F_1}^{kl}-\sigma_{F_2}^{kl}\|_{1+\delta',m'}\nonumber\\
&\leqslant& C\|G\|_{1+\delta',r'}\cdot\|F_1-F_2\|_{2+\delta',r''}.
\end{eqnarray}
For the second term, in view of $\gamma_{F_2}^{ij}\cdot H\in{\mathbb D}_{(1+\delta)-}^{\infty-}$, we obtain
\begin{eqnarray}\label{B2}
\Lambda_2&\leqslant& C\|\gamma_{F_2}^{ij}\cdot H\cdot G\|_{1+\delta',n}\cdot\|DF_1^j-DF_2^j\|_{1+\delta',m}\nonumber\\
&\leqslant& C\|\gamma_{F_2}^{ij}\cdot H\|_{1+\delta',r'^c}\cdot\|G\|_{1+\delta',r'}\cdot\|DF_1^j-DF_2^j\|_{1+\delta',m}\nonumber\\
&\leqslant& C\|G\|_{1+\delta',r'}\cdot\|F_1-F_2\|_{2+\delta',m}.
\end{eqnarray}
Hence combining (\ref{B1}) and (\ref{B2}), we have
\begin{eqnarray*}
\|H_i(F_1,H\cdot G)-H_i(F_2,H\cdot G)\|_{\delta',r}\leqslant C\|G\|_{1+\delta',r'}\cdot\|F_1-F_2\|_{2+\delta',r''}.
\end{eqnarray*}
Thus the proof is completed.
\end{proof}

With the above preparation we can establish the following results which will play a crucial role
in the proof of Theorem \ref{comparison theorem}.
\begin{lemma}\label{comparison lemma}
Let $H,H_1,H_2\in B\rightarrow[0,1]$ and $F_1,F_2\in B\rightarrow{\mathbb R}^d$ satisfy all the conditions (i), (ii) and (iii) of
Theorem \ref{comparison theorem}. Then, for every $0<\delta''<\delta'<
\delta$ and $p$, $p'$, $p''$, $p'''$, $r_1$, $r_2$ satisfying $1<p<p'<p''
<p'''<\infty$, $r_1>pp'/(p'-p)$ and $r_2>r_3=p'p''/(p''-p')$,
we can find a positive constant $C$ such that
\begin{eqnarray*}
&&\|g\circ F_1\cdot H-g\circ F_2\cdot H\|_{-(1+{\delta'}),p}\nonumber\\
&\leqslant& C\|g\|_{-(1+\delta''),p''}\cdot\|F_1-F_2\|_{2+\delta',r_1}+C\|g\|_{-\delta',p'''}
\cdot\|F_1-F_2\|_{1,r_2}\nonumber\\
&&{}+C\|g\|_{-\delta'',p'''}\cdot\|F_1-F_2\|_{\delta',r_3}
\end{eqnarray*}
for every $g\in S({\mathbb R}^d)$. The positive constant $C$ may depend on $F_1$, $F_2$,
$\delta'$, $\delta''$, $p$, $p'$, $p''$, $p'''$, $r_1$, $r_2$.
\end{lemma}

\begin{proof}
In what follows, $q=p/(p-1)$, $q'=p'/(p'-1)$, $q''=p''/(p''-1)$.
First using the integration by parts formula, we have for $G\in{\mathbb D}_{1+\delta}^{\infty-}$,
\begin{eqnarray*}
&&|E[(\partial_ig)\circ F_1\cdot H\cdot G-(\partial_ig)\circ F_2\cdot H\cdot G]|\\
&=&|E[g\circ F_1\cdot H_1\cdot H_i(F_1,H\cdot G)-g\circ F_2\cdot H_1\cdot H_i(F_2,H\cdot G)]|\\
&\leqslant& |E[g\circ F_1\cdot H_1\cdot H_i(F_1,H\cdot G)-g\circ F_2\cdot H_1\cdot H_i(F_1,H\cdot G)]|\\
&&{}+|E[g\circ F_2\cdot H_1\cdot H_i(F_1,H\cdot G)-g\circ F_2\cdot H_1\cdot H_i(F_2,H\cdot G)]|\\
&\leqslant& \|g\circ F_1\cdot H_1-g\circ F_2\cdot H_1\|_{-\delta',p'}\cdot\|H_i(F_1,H\cdot G)\|_{\delta',q'}\\
&&{}+\|g\circ F_2\cdot H_1\|_{-\delta',p'}\cdot\|H_i(F_1,H\cdot G)-H_i(F_2,H\cdot G)\|_{\delta',q'}.
\end{eqnarray*}
Since
\begin{eqnarray*}
H_i(F_1,H\cdot G)=\sum_{j=1}^dD^*(\gamma_{F_1}^{ij}\cdot H\cdot G\cdot DF_1^j),
\end{eqnarray*}
by $\gamma_{F_1}^{ij}\cdot H\in {\mathbb D}_{(1+{\delta})-}^{\infty-}$ and $D^*:{\mathbb D}_{1+\alpha}^{p}({\mathbb H})
\rightarrow {\mathbb D}_{\alpha}^p$ is continuous for every $p$ and every $\alpha$, the map
$G\rightarrow H_i(F_1,H\cdot G)$ can be extended to a continuous operator ${\mathbb D}_{(1+{\delta})-}^q\rightarrow
{\mathbb D}_{{\delta}-}^{q-}:=\bigcap_{q'<q}\mathbb D_{{\delta}-}^{q'}$ for every $q$ and every ${\delta}$, that is, we can find, for every $1<q'<q$,
$0<\delta'<{\delta}$, a positive constant $C=C(q,q',\delta';F_1)$ such that
\begin{eqnarray*}
\|H_i(F_1,H\cdot G)\|_{\delta',q'}\leqslant C\|G\|_{1+{\delta}',q}.
\end{eqnarray*}
Then by Lemma \ref{key estimate}, we have
\begin{eqnarray}\label{intermediate results}
&&\|(\partial_ig)\circ F_1\cdot H-(\partial_ig)\circ F_2\cdot H\|_{-(1+{\delta'}),p}\nonumber\\
&=&\sup\{|E[(\partial_ig)\circ F_1\cdot H\cdot G-(\partial_ig)\circ F_2\cdot H\cdot G]|:\|G\|_{1+{\delta'},q}\leqslant 1\}\nonumber\\
&\leqslant&\sup\{\|g\circ F_1\cdot H_1-g\circ F_2\cdot H_1\|_{-\delta',p'}\cdot\|H_i(F_1,H\cdot G)\|_{\delta',q'}:
\|H_i(F_1,H\cdot G)\|_{\delta',q'}\leqslant C\}\nonumber\\
&&\quad{}+\sup\{\|g\circ F_2\cdot H_1\|_{-\delta',p'}\cdot\|H_i(F_1,H\cdot G)-H_i(F_2,H\cdot G)\|_{\delta',q'}:
\|G\|_{1+{\delta'},q}\leqslant 1\}\nonumber\\
&\leqslant& C\|g\circ F_1\cdot H_1-g\circ F_2\cdot H_1\|_{-\delta',p'}+C\|g\circ F_2\cdot H_1\|_{-\delta',p'}\cdot\|F_1-F_2\|_{2+\delta',r_1}.
\end{eqnarray}
Since any $g\in S({\mathbb R}^d)$ can be written in the form
\begin{eqnarray*}
g=\sum_{n=0}^1\Sigma'\pm\partial_{i_1}\cdots\partial_{i_{2n}}(1-\Delta)^{-1}g,
\end{eqnarray*}
where $\Sigma'$ is a certain sum over indices $(i_1,\cdots,i_{2n})$,
by (\ref{intermediate results}), we have
\begin{eqnarray*}
&&\|g\circ F_1\cdot H-g\circ F_2\cdot H\|_{-(1+{\delta'}),p}\\
&=&\|\sum_{n=0}^1\Sigma'\pm\partial_{i_1}\cdots\partial_{i_{2n}}
(1-\Delta)^{-1}g\circ F_1\cdot H\\
&&-\sum_{n=0}^1\Sigma'\pm\partial_{i_1}\cdots\partial_{i_{2n}}(1-\Delta)^{-1}g\circ F_2\cdot H\|_{-(1+{\delta'}),p}\\
&\leqslant& C\Sigma'\|\partial_{i_2}(1-\Delta)^{-1}g\circ F_1\cdot H_1-\partial_{i_2}(1-\Delta)^{-1}
g\circ F_2\cdot H_1\|_{-\delta',p'}\\\
&&{}+C\Sigma'\|\partial_{i_2}(1-\Delta)^{-1}g\circ F_2\cdot H_1\|_{-\delta',p'}\cdot\|F_1-F_2\|_{2+\delta',r_1}\\
&=:&\Xi_1+\Xi_2.
\end{eqnarray*}
By Lemma \ref{local version2} and the ${\mathbb L}_{\alpha}^p({\mathbb R}^d)$ boundedness of $\partial_{i_2}(1-\Delta)^{-\frac 1 2}$, we have
\begin{eqnarray*}
\|\partial_{i_2}(1-\Delta)^{-1}g\circ F_2\cdot H_1\|_{-\delta',p'}
\leqslant C\|\partial_{i_2}(1-\Delta)^{-1}g\|_{-\delta'',p''}\leqslant C\|g\|_{-(1+\delta''),p''}.
\end{eqnarray*}
Then we obtain
\begin{eqnarray}\label{t1}
\Xi_2\leqslant C\|g\|_{-(1+\delta''),p''}\cdot\|F_1-F_2\|_{2+\delta',r_1}.
\end{eqnarray}
Now we deal with the first term $\Xi_1$. We split $\Xi_1$ into two parts
\begin{eqnarray*}
&&\|\partial_{i_2}(1-\Delta)^{-1}g\circ F_1\cdot H_1-\partial_{i_2}(1-\Delta)^{-1}g\circ F_2\cdot H_1\|_{-\delta',p'}\\
&=&\sup\{|E[\partial_{i_2}(1-\Delta)^{-1}g\circ F_1\cdot H_1\cdot G-\partial_{i_2}(1-\Delta)^{-1}g\circ F_2\cdot H_1\cdot G]|:
\|G\|_{\delta',q'}\leqslant 1\}\\
&\leqslant&\sup\{|E[(\partial_{i_2}(1-\Delta)^{-1}g\circ F_1\cdot H_1\cdot G-\partial_{i_2}(1-\Delta)^{-1}g\circ F_2\cdot H_1\cdot G)
\cdot(1-\Psi(R_{F_1,F_2}))]|\\
&&:\|G\|_{\delta',q'}\leqslant 1\}\\
&&{}+\sup\{|E[(\partial_{i_2}(1-\Delta)^{-1}g\circ F_1\cdot H_1\cdot G-\partial_{i_2}(1-\Delta)^{-1}
g\circ F_2\cdot H_1\cdot G)\cdot\Psi(R_{F_1,F_2})]|\\
&&:\|G\|_{\delta',q'}\leqslant 1\}\\
&=:&\Xi_3+\Xi_4,
\end{eqnarray*}
where $\Psi$ and $R_{F_1,F_2}$ are defined in Section 2.
First we consider the term $\Xi_3$. In view of
$(\det(\Sigma_{F_1}))^{-1}\cdot 1_{\{H_2\neq 0\}}\in{\mathbb L}^{\infty-}$, $F_1\in{\mathbb D}_{2+\delta}^{\infty-}({\mathbb R}^d)$, we have
\begin{eqnarray*}
&&P(\{\Psi(R_{F_1,F_2})\neq 1\}\cap\{H_2\neq 0\})\\
&\leqslant& P(\{((\det(\Sigma_{F_1}))^{-1})^{1/2}(1+\|\Sigma_{F_1}\|_2^2)^{(d-1)/4}
\|D(F_1-F_2)\|_{\mathbb H}>\frac{1}{2\sqrt{2}}\}\cap\{H_2\neq 0\})\\
&\leqslant&(2\sqrt{2})^{r_3'}E[(((\det(\Sigma_{F_1}))^{-1})^{1/2}(1+\|\Sigma_{F_1}\|_2^2)^{(d-1)/4})^{r_3'}\|D(F_1-F_2)\|_{\mathbb H}
^{r_3'}\cdot 1_{\{H_2\neq0\}}]\\
&\leqslant&CE[\|DF_1-DF_2\|_{\mathbb H}^{r_2}]^{r_3'/r_2},
\end{eqnarray*}
where $r_3<r_3'<r_2$.
Then by Remark \ref{localness}, we obtain
\begin{eqnarray*}
&&|E[(\partial_{i_2}(1-\Delta)^{-1}g\circ F_1\cdot H_1\cdot G-\partial_{i_2}(1-\Delta)^{-1}g\circ F_2\cdot H_1\cdot G)
\cdot(1-\Psi(R_{F_1,F_2}))]|\\
&\leqslant& |E[\partial_{i_2}(1-\Delta)^{-1}g\circ F_1\cdot H_1\cdot H_2\cdot G\cdot(1-\Psi(R_{F_1,F_2}))]|\\
&&{}+|E[\partial_{i_2}
(1-\Delta)^{-1}g\circ F_2\cdot H_1\cdot H_2\cdot G\cdot(1-\Psi(R_{F_1,F_2}))]|\\
&\leqslant& C\|\partial_{i_2}(1-\Delta)^{-1}g\circ F_1\cdot H_1\|_{-k',p''}\cdot\|H_2\cdot G\cdot(1-\Psi(R_{F_1,F_2}))\|_{k',q''}\\
&&{}+C\|\partial_{i_2}(1-\Delta)^{-1}g\circ F_2\cdot H_1\|_{-k',p''}\cdot\|H_2\cdot G\cdot(1-\Psi(R_{F_1,F_2}))\|_{k',q''}\\
&\leqslant& C\|\partial_{i_2}(1-\Delta)^{-1}g\circ F_1\cdot H_1\|_{-k',p''}\cdot\|G\|_{k',q'}\cdot\|H_2\cdot(1-\Psi(R_{F_1,F_2}))\|_{k',r_3}\\
&&{}+C\|\partial_{i_2}(1-\Delta)^{-1}g\circ F_2\cdot H_1\|_{-k',p''}\cdot\|G\|_{k',q'}\cdot\|H_2\cdot(1-\Psi(R_{F_1,F_2}))\|_{k',r_3}\\
&\leqslant& C\|\partial_{i_2}(1-\Delta)^{-1}g\circ F_1\cdot H_1\|_{-k',p''}\cdot\|G\|_{k',q'}\cdot P(\{\Psi(R_{F_1,F_2})\neq 1\}\cap\{H_2\neq 0\})^{1/r_3'}\\
&&{}+C\|\partial_{i_2}(1-\Delta)^{-1}g\circ F_2\cdot H_1\|_{-k',p''}\cdot\|G\|_{k',q'}\cdot P(\{\Psi(R_{F_1,F_2})\neq 1\}\cap\{H_2\neq 0\})^{1/r_3'}\\
&\leqslant& C\|\partial_{i_2}(1-\Delta)^{-1}g\circ F_1\cdot H_1\|_{-k',p''}\cdot\|G\|_{k',q'}\cdot\|F_1-F_2\|_{1,r_2}\\
&&{}+C\|\partial_{i_2}(1-\Delta)^{-1}g\circ F_2\cdot H_1\|_{-k',p''}\cdot\|G\|_{k',q'}\cdot\|F_1-F_2\|_{1,r_2},
\end{eqnarray*}
where $k'<\delta'\leqslant k'+1$. By Lemma \ref{local version2} and Remark \ref{bally remark}, it is obvious that
\begin{eqnarray*}
\|\partial_{i_2}(1-\Delta)^{-1}g\circ F_i\cdot H_1\|_{-k',p''}\leqslant C\|\partial_{i_2}(1-\Delta)^{-1}g\|_{-k',p'''}\leqslant C\|g\|_{-(k'+1),p'''}.
\end{eqnarray*}
Consequently
\begin{eqnarray}\label{t2}
\Xi_3\leqslant C\|g\|_{-(k'+1),p'''}\cdot\|F_1-F_2\|_{1,r_2}\leqslant C\|g\|_{-\delta',p'''}\cdot\|F_1-F_2\|_{1,r_2}.
\end{eqnarray}
Next we consider the second term $\Xi_4$.
First by (\ref{important result1}), we have
\begin{eqnarray*}
\{\Psi_1(R_{F_1,F_2})\neq 0\}\subset \{\det(\Sigma_{(F_2+t(F_1-F_2))})\geqslant(1-\frac{\sqrt{2}}{2})^{2d}
\frac{(\det(\Sigma_{F_1}))^d}{\|\Sigma_{F_1}\|^{d(d-1)}}\}\quad\text{a.s.}.
\end{eqnarray*}
Then by $(\det(\Sigma_{F_1}))^{-1}\cdot 1_{\{H_2\neq 0\}}\in{\mathbb L}^{\infty-}$, $F_1\in{\mathbb D}_{2+\delta}^{\infty-}({\mathbb R}^d)$,
we obtain that $F_2+t(F_1-F_2)$ is nondegenerate a.s. on the set
$\{H_2\neq0\}\cap\{\Psi_1(R_{F_1,F_2})\neq 0\}$, uniformly in $t\in[0,1]$.
Besides, by Remark \ref{composite} it is obvious that $F_2+t(F_1-F_2)\in{\mathbb D}_{2+\delta}^{\infty-}({\mathbb R}^d)$, $\Psi(R_{F_1,F_2})\cdot H_1
\in{\mathbb D}_{(1+\delta)-}^{\infty-}$ and $\Psi_1(R_{F_1,F_2})\cdot H_2\in{\mathbb D}_{(1+\delta)-}^{\infty-}$, and hence
$F_2+t(F_1-F_2)\in{\mathbb D}_{2+\delta'}^{\infty-}({\mathbb R}^d)$, $\Psi(R_{F_1,F_2})\cdot H_1\in{\mathbb D}_{k'+1}^{\infty-}$ and
$\Psi_1(R_{F_1,F_2})\cdot H_2\in{\mathbb D}_{k'+1}^{\infty-}$. Therefore, $F_2+t(F_1-F_2)$, $\Psi(R_{F_1,F_2})\cdot H_1$ and
$\Psi_1(R_{F_1,F_2})\cdot H_2$ satisfy all the conditions in
Lemma \ref{local version2} by condition \eqref{bally example}. Consequently, applying mean value theorem, Remark \ref{bally remark} (2) and
Lemma \ref{local version2}, we have
\begin{eqnarray}\label{t3}
&&\Xi_4\nonumber\\
&=&\|\partial_{i_2}(1-\Delta)^{-1}g\circ F_1\cdot H_1\cdot\Psi(R_{F_1,F_2})-\partial_{i_2}(1-\Delta)^{-1}g\circ
F_2\cdot H_1\cdot\Psi(R_{F_1,F_2})\|_{-\delta',p'}\nonumber\\
&\leqslant&\int_0^1\|\partial_i\partial_{i_2}(1-\Delta)^{-1}g\circ(F_2+t(F_1-F_2))\cdot(F_1^i-F_2^i)\cdot H_1\cdot
\Psi(R_{F_1,F_2})\|_{-\delta',p'}dt\nonumber\\
&\leqslant& C\int_0^1\|\partial_i\partial_{i_2}(1-\Delta)^{-1}g\circ(F_2+t(F_1-F_2))\cdot
H_1\cdot\Psi(R_{F_1,F_2})\|_{-\delta',p''}\cdot\|F_1^i-F_2^i\|_{\delta',r_3}dt\nonumber\\
&\leqslant& C\|\partial_i\partial_{i_2}(1-\Delta)^{-1}g\|_{-\delta'',p'''}\cdot\|F_1-F_2\|_{\delta',r_3},
\end{eqnarray}
and, noting the ${\mathbb L}_{\alpha}^p({\mathbb R}^d)$ boundedness of $\partial_{i_1}\partial_{i_2}(1-\Delta)^{-1}$, this is further
dominated by $C\|g\|_{-\delta'',p'''}\cdot\|F_1-F_2\|_{\delta',r_3}$.
Thus combining (\ref{t1}), (\ref{t2}) and (\ref{t3}), we have
\begin{eqnarray*}
&&\|g\circ F_1\cdot H-g\circ F_2\cdot H\|_{-(1+{\delta'}),p}\\
&\leqslant& C_1\|g\|_{-(1+\delta''),p''}\cdot\|F_1-F_2\|_{2+\delta',r_1}+ C_2\|g\|_{-\delta',p'''}\cdot\|F_1-F_2\|_{1,r_2}\\
&&{}+C_3\|g\|_{-\delta'',p'''}\cdot\|F_1-F_2\|_{\delta',r_3}.
\end{eqnarray*}
This completes the proof of the lemma.
\end{proof}

\begin{remark}
In the proof of Lemma \ref{comparison lemma}, the reason why we make use of the integration by parts formula rather
than a direct use of the mean value theorem, is to ensure that $\Psi(R_{F_1,F_2})\cdot H_1$ satisfies the condition (ii) of
Lemma \ref{local version2}.
\end{remark}

\begin{proof}[Proof of Theorem \ref{comparison theorem}]
In what follows, $q=p/(p-1)$, $q'=p'/(p'-1)$, $q''=p''/(p''-1)$. Since $0<\beta<\alpha\wedge(1+\delta)-1-d/q$,
we can choose $1<q'''<q''$ and $0<\delta''<\delta'<\delta$ such that
$\delta''-\beta>d/q'''$. Hence by Lemma \ref{w1}, we have
\begin{eqnarray}\label{real interpolation}
x\rightarrow(1-\Delta)^{\beta/2}\delta_x\in {\mathbb L}_{-\delta''}^{p'''}({\mathbb R}^d)
\end{eqnarray}
is bounded and continuous, where $1/p'''+1/q'''=1$.
Hence by Lemma \ref{comparison lemma} and (\ref{real interpolation}), we can find $C>0$ such that
\begin{eqnarray*}
&&\|(1-\Delta)^{\beta/2}\delta_x\circ F_1\cdot H-(1-\Delta)^{\beta/2}\delta_x\circ F_2\cdot H\|_{-\alpha,p}\\
&\leqslant& \|(1-\Delta)^{\beta/2}\delta_x\circ F_1\cdot H-(1-\Delta)^{\beta/2}\delta_x\circ F_2\cdot H\|_{-(1+{\delta'}),p}\\
&\leqslant& C\|(1-\Delta)^{\beta/2}\delta_x\|_{-(1+\delta''),p''}\cdot\|F_1-F_2\|_{2+\delta',r_1}\\
&&\quad{}+ C\|(1-\Delta)^{\beta/2}\delta_x\|_{-\delta',p'''}\cdot\|F_1-F_2\|_{1,r_2}\\
&&\quad\quad{}+C\|(1-\Delta)^{\beta/2}\delta_x\|_{-\delta'',p'''}\cdot\|F_1-F_2\|_{\delta',r_3}\\
&\leqslant& C\|F_1-F_2\|_{2+\delta',r_1}+ C\|F_1-F_2\|_{1,r_2}+C\|F_1-F_2\|_{\delta',r_3}.
\end{eqnarray*}
The proof is thus completed.
\end{proof}

Similar to the proof of Theorem \ref{comparison theorem}, by Theorem \ref{inverse worse},
we have the following theorem which shows that the condition $H,H_1,H_2\in{\mathbb D}_{k+2}^{\infty-}$ can be improved to
$H,H_1,H_2\in{\mathbb D}_{(1+\delta)-}^{\infty-}$.
\begin{theorem}\label{comparison theorem worse}
Let $\delta=k+\sigma$, $k\in{\mathbb N}$, $0<\sigma\leqslant1$. Suppose that $H,H_1,H_2:B\rightarrow[0,1]$ and $F_1,F_2:B\rightarrow {\mathbb R}^d$ are
Wiener functionals such that
\begin{itemize}
  \item[(i)] $F_1,F_2\in{\mathbb D}_{2+\delta}^{\infty-}({\mathbb R}^d)$;
  \item[(ii)]$H,H_1,H_2\in{\mathbb D}_{(1+\delta)-}^{\infty-}$ with
  \begin{eqnarray*}
  1+E_{H_2}[\|D\ln H_2\|_{\mathbb H}^p]<\infty\quad\text{for every}\quad p\geq1,
  \end{eqnarray*}
  such that $H_1=1$ on the set $\{H\neq0\}$ and $H_2=1$ on the set $\{H_1\neq0\}$;
  \item[(iii)] there is a measurable set $A$ such that $H_21_A=0$ and $F_1,F_2$ are nondegenerate a.s. on the set $A^c$.
\end{itemize}
Then for every $p$, $p'$, $p''$, $r_1$, $r_2$ and $r_3$ satisfying $1<p<p'<p''<\infty$, $r_1>pp'/(p'-p)$ and $r_2>r_3=p'p''/(p''-p')$, and for
\begin{eqnarray*}
0<\beta<\alpha\wedge(1+k+\frac{\sigma}{2^{k+2}})-1-\frac{d(p-1)}{p}, \beta+\frac{d(p''-1)}{p''}<\delta'<k+\frac{\sigma}{2^{k+2}}<\delta,
\delta'\leqslant\alpha-1,
\end{eqnarray*}
we can find a positive constant $C$ which may depend on $F_1$, $F_2$, $\alpha$, $\beta$, $\delta$, $\delta'$, $p$, $p'$, $p''$, $r_1$, $r_2$
 and $r_3$ such that
\begin{eqnarray*}
&&\|(1-\Delta)^{\beta/2}\delta_x\circ F_1\cdot H-(1-\Delta)^{\beta/2}\delta_x\circ F_2\cdot H\|_{-\alpha,p}\nonumber\\
&\leqslant& C\|F_1-F_2\|_{2+\delta',r_1}+C\|F_1-F_2\|_{1,r_2}+C\|F_1-F_2\|_{\delta',r_3}.
\end{eqnarray*}
\end{theorem}

\section{Convergence rate of density of the Euler scheme for non-Markovian stochastic differential equations:
Applications of Theorem \ref{comparison theorem}}
In this section we study the convergence rate of density of the Euler scheme for non-Markovian stochastic differential equations. Before proceeding, we introduce some notations and notions. Given compact metric spaces $M_1,\cdots,M_d$ and real separable Hilbert spaces $E_1,\cdots,E_d$ and $E$, let $k\geqslant 1$ be an integer, we denote by $C_p^k(C(M_1;E_1)\otimes\cdots\otimes C(M_d;E_d);E)$ the class of continuous $F:C(M_1;E_1)\otimes\cdots\otimes C(M_d;E_d)\rightarrow E$ such that for any multi-index $\alpha=(\alpha_1,\cdots,\alpha_d)$ with $1\leqslant|\alpha|=\alpha_1+\cdots+\alpha_d\leqslant k$
and any $g_j^i\in C(M_i;E_i)$, $1\leqslant i\leqslant d$, $0\leqslant j\leqslant \alpha_i$, the map
\begin{eqnarray*}
(x_1^1,\cdots,x_{\alpha_1}^1,\cdots,x_1^d,\cdots,x_{
\alpha_d}^d)\in{\mathbb R}^{|\alpha|}\mapsto F\left(
                                       \begin{array}{c}
                                         g_0^1+\sum_{j=1}^{\alpha_1}x_j^1g_j^1 \\
                                         \vdots \\
                                         g_0^d+\sum_{j=1}^{\alpha_d}x_j^dg_j^d \\
                                       \end{array}
                                     \right)~\text{is}~C^k;
\end{eqnarray*}
there is a continuous
\begin{eqnarray*}
F^{(\alpha)}:C(M_1;E_1)\otimes\cdots\otimes C(M_d;E_d)\rightarrow Hom (C(M^{\alpha};E^{\otimes\alpha});E)
\end{eqnarray*}
for which
\begin{eqnarray}\label{high order derivative}
&&\frac{\partial^{\alpha}}{\partial x_1^1\cdots\partial x_{\alpha_1}^1\cdots\partial x_1^d\cdots\partial x_{\alpha_d}^d}F\left(
                                                                          \begin{array}{c}
                                                                            g_0^1+\sum_{j=1}^{\alpha_1}x_j^1g_j^1 \\
                                                                            \vdots \\
                                                                            g_0^d+\sum_{j=1}^{\alpha_d}x_j^dg_j^d \\
                                                                          \end{array}
                                                                        \right)|_{x_1^1=\cdots=x_{\alpha_d}^d=0}\nonumber\\
&=&F^{(\alpha)}\left(
                \begin{array}{c}
                  g_0^1 \\
                  \vdots \\
                  g_0^d \\
                \end{array}
              \right)(g_1^1(*_{1,1})\otimes\cdots\otimes g_{\alpha_1}^1(*_{1,\alpha_1})\otimes\cdots\otimes g_1^d(*_{d,1})
              \otimes\cdots\otimes g_{\alpha_d}^d(*_{d,\alpha_d}));
\end{eqnarray}
and, for any $0\leqslant|\alpha|\leqslant k$, there exist $C_{\alpha}<\infty$ and $\lambda_{\alpha}<\infty$ such that
\begin{eqnarray*}
\|F^{(\alpha)}\left(
              \begin{array}{c}
                g^1 \\
                \vdots \\
                g^d \\
              \end{array}
            \right)\|_{Hom(C(M^{\alpha};E^{\otimes\alpha});E)}\leqslant C_{\alpha}(1+\sum_{i=0}^d\|g^i\|_{C(M_i;E_i)})^{\lambda_{\alpha}}
\end{eqnarray*}
for $g^{i}\in C(M_i;E_i)$, $i=1,\cdots,d$. Similarly, let $k\geqslant 1$ be an integer, we denote by $C_b^k(C(M_1;E_1)\otimes\cdots\otimes C(M_d;E_d);E)$ the class of continuous $F:C(M_1;E_1)\otimes\cdots\otimes C(M_d;E_d)\rightarrow E$ such that for any multi-index $\alpha=(\alpha_1,\cdots,\alpha_d)$ with $1\leqslant|\alpha|=\alpha_1+\cdots+\alpha_d\leqslant k$
and any $g_j^i\in C(M_i;E_i)$, $1\leqslant i\leqslant d$, $0\leqslant j\leqslant \alpha_i$, the map
\begin{eqnarray*}
(x_1^1,\cdots,x_{\alpha_1}^1,\cdots,x_1^d,\cdots,x_{
\alpha_d}^d)\in{\mathbb R}^{|\alpha|}\mapsto F\left(
                                       \begin{array}{c}
                                         g_0^1+\sum_{j=1}^{\alpha_1}x_j^1g_j^1 \\
                                         \vdots \\
                                         g_0^d+\sum_{j=1}^{\alpha_d}x_j^dg_j^d \\
                                       \end{array}
                                     \right)~\text{is}~C^k;
\end{eqnarray*}
there is a continuous
\begin{eqnarray*}
F^{(\alpha)}:C(M_1;E_1)\otimes\cdots\otimes C(M_d;E_d)\rightarrow Hom (C(M^{\alpha};E^{\otimes\alpha});E)
\end{eqnarray*}
for which
\begin{eqnarray*}
&&\frac{\partial^{\alpha}}{\partial x_1^1\cdots\partial x_{\alpha_1}^1\cdots\partial x_1^d\cdots\partial x_{\alpha_d}^d}F\left(
                                                                          \begin{array}{c}
                                                                            g_0^1+\sum_{j=1}^{\alpha_1}x_j^1g_j^1 \\
                                                                            \vdots \\
                                                                            g_0^d+\sum_{j=1}^{\alpha_d}x_j^dg_j^d \\
                                                                          \end{array}
                                                                        \right)|_{x_1^1=\cdots=x_{\alpha_d}^d=0}\nonumber\\
&=&F^{(\alpha)}\left(
                \begin{array}{c}
                  g_0^1 \\
                  \vdots \\
                  g_0^d \\
                \end{array}
              \right)(g_1^1(*_{1,1})\otimes\cdots\otimes g_{\alpha_1}^1(*_{1,\alpha_1})\otimes\cdots\otimes g_1^d(*_{d,1})
              \otimes\cdots\otimes g_{\alpha_d}^d(*_{d,\alpha_d}));
\end{eqnarray*}
and, for any $0\leqslant|\alpha|\leqslant k$, there exists $C_{\alpha}<\infty$ such that
\begin{eqnarray*}
\|F^{(\alpha)}\left(
              \begin{array}{c}
                g^1 \\
                \vdots \\
                g^d \\
              \end{array}
            \right)\|_{Hom(C(M^{\alpha};E^{\otimes\alpha});E)}\leqslant C_{\alpha}
\end{eqnarray*}
for $g^{i}\in C(M_i;E_i)$, $i=1,\cdots,d$, furnished with the following norm:
\begin{eqnarray*}
\|F\|_{C_b^k(C(M_1;E_1)\otimes\cdots\otimes C(M_d;E_d);E)}=\sum_{|\alpha|=0}^k\sup_{g^{i}\in C(M_i;E_i)\atop i=1,\cdots,d}\|F^{(\alpha)}\left(
              \begin{array}{c}
                g^1 \\
                \vdots \\
                g^d \\
              \end{array}
            \right)\|_{Hom(C(M^{\alpha};E^{\otimes\alpha});E)}.
\end{eqnarray*}

For any $p>1$, let ${\mathcal G}_0^p$ be the class of continuous progressively measurable functions: $\xi:[0,T]\times B\rightarrow E$ such that
\begin{eqnarray*}
\|\xi\|_{0,p,T;E}:=\|\sup_{0\leqslant t\leqslant T}\|\xi(t)\|_E\|_{{\mathbb L}^p}<\infty.
\end{eqnarray*}
For every integer $k\geqslant 1$ and any $p>1$, let ${\mathcal G}_k^p$ be the class of $\xi\in {\mathcal G}_0^p$ such that $\xi(t)\in{\mathbb D}_j^p(E)$ for each $t\in [0,T]$ and $j=1,\cdots,k$, equipped with the following norm
\begin{eqnarray}\label{soblev norm}
\|\xi\|_{k,p,T;E}:=\sum_{0\leqslant j\leqslant k}\|\sup_{0\leqslant t\leqslant T}\|D^j\xi(t)\|_{{\mathbb H}^{\otimes j}\otimes E}\|_{{\mathbb L}^p}<\infty.
\end{eqnarray}
Finally, we define ${\mathcal G}_{\infty}^{\infty-}$ by
\begin{eqnarray*}
{\mathcal G}_{\infty}^{\infty-}=\bigcap_{k\geqslant 0}\bigcap_{1<p<\infty}{\mathcal G}_k^p.
\end{eqnarray*}
Now we introduced the following fractional order Sobolev space ${\mathcal G}_{\alpha}^p$.
\begin{definition}
For every integer $k\geq 0$, any $0<\alpha<1$ and any $1<p<\infty$, we define
\begin{eqnarray*}
{\mathcal G}_{k+\alpha}^p=({\mathcal G}_k^p,{\mathcal G}_{k+1}^p)_{\alpha,p}.
\end{eqnarray*}
\end{definition}
We shall use the following norms which was given by Peetre's $K$-method.
\begin{eqnarray}\label{fractional regularity soblev norm}
\|\xi\|_{k+\alpha,p,T;E}:=\big[\int_0^1[\epsilon^{-\alpha}K(\epsilon,p,\xi)]^p\frac{d\epsilon}{\epsilon}\big]^{\frac{1}{p}},
\end{eqnarray}
where
\begin{eqnarray*}
K(\epsilon,p,\xi):=\inf\{\|\xi_1\|_{k,p,T;E}+\epsilon\|\xi_2\|_{k+1,p,T;E},\xi_1+\xi_2=\xi,\xi_1,\xi_2\in{\mathcal G}_k^p\}.
\end{eqnarray*}

The following definition is taken from \cite{KS}.
\begin{definition}
Let $k$ be an integer and $E_1$ and $E_2$ be real separable Hilbert spaces. We say that a function $F:[0,T]\times C([0,T];E_1)\rightarrow E_2$ satisfies $(A_{k,E_1,E_2})$ if
\begin{enumerate}
\item[(i)]
$F$ is measurable, and for each $t\in [0,T]$, there is an
\begin{eqnarray*}
F(t)\in C_p^{k}(C([0,t];E_1);E_2)
\end{eqnarray*}
such that $F(t,\psi)=F(t)(\psi|_{[0,t]})$ for all $\psi\in C([0,T];E_1)$;
\item[(ii)]
for each integer $0\leqslant n\leqslant k$, there exist $C_n<\infty$ and $\gamma_n$, where $\gamma_n=1$ for $n=0$ and $\gamma_n=0$ for $n\geqslant 1$, such that
\begin{eqnarray}\label{a1}
\|F(t)^{(n)}(\psi)\|_{Hom(C([0,t]^n;E_1^{\otimes n});E_2)}\leqslant C_n(1+\|\psi\|_{C([0,t];E_1)})^{\gamma_n}
\end{eqnarray}
for all $t\in[0, T]$ and $\psi\in C([0,t];E_1)$;
\item[(iii)]
for each integer $0\leqslant n\leqslant k$, there exists $C_n<\infty$ such that
\begin{eqnarray}\label{a2}
\|F(t)^{(n)}(\psi)(g|_{[0,t]^n})-F(s)^{(n)}(\psi)(g|_{[0,s]^n})\|_{E_2}
\leqslant C_n|t-s|^{1/2}(1+\|\psi\|_{C([0,T];E_1)})
\end{eqnarray}
for all $t,s\in[0, T]$, $\psi\in C([0,T];E_1)$ and $g\in C([0,T]^n;E_1^{\otimes n})$.
\end{enumerate}
\end{definition}

We also need the following the following definition.
\begin{definition}
Let $\delta=k+\theta$, $k\in {\mathbb N}$, $0<\theta<1$, $p>1$ and $E_1$ and $E_2$ be real separable Hilbert spaces. We say that a function $F:[0,T]\times C([0,T];E_1)\rightarrow E_2$ satisfies $(A_{1+\delta,E_1,E_2})$ (or $(A_{2+\delta,E_1,E_2})$) if
\begin{enumerate}
\item[(i)]
$F$ is measurable, and for each $t\in [0,T]$, there is an
\begin{eqnarray*}
&&F(t)\in(C_b^{1+k}(C([0,t];E_1);E_2),C_b^{2+k}(C([0,t];E_1);E_2))_{\theta,p}\\
&&(\text{or}\quad F(t)\in(C_b^{2+k}(C([0,t];E_1);E_2),C_b^{3+k}(C([0,t];E_1);E_2))_{\theta,p})
\end{eqnarray*}
such that $F(t,\psi)=F(t)(\psi|_{[0,t]})$ for all $\psi\in C([0,T];E_1)$, where the norm
\begin{eqnarray*}
&&\|F(t)\|_{(C_b^{1+k}(C([0,t];E_1);E_2),C_b^{2+k}(C([0,t];E_1);E_2))_{\theta,p}}\\
&&(\text{or}\quad\|F(t)\|_{(C_b^{2+k}(C([0,t];E_1);E_2),C_b^{3+k}(C([0,t];E_1);E_2))_{\theta,p}})
\end{eqnarray*}
is controlled by a constant which is independent of $t$;
\item[(ii)]
for each integer $0\leqslant n\leqslant 1+k$ (or $2+k$), there exists $C_n<\infty$ such that
\begin{eqnarray*}
\|F(t)^{(n)}(\psi)(g|_{[0,t]^n})-F(s)^{(n)}(\psi)(g|_{[0,s]^n})\|_{E_2}
\leqslant C_n|t-s|^{1/2}(1+\|\psi\|_{C([0,T];E_1)})
\end{eqnarray*}
for all $t,s\in[0, T]$, $\psi\in C([0,T];E_1)$ and $g\in C([0,T]^n;E_1^{\otimes n})$.
\end{enumerate}
\end{definition}

Now we introduce the Euler scheme for SDE (\ref{SDE2}). Let $T>0$ be a fixed time horizon, and let $T/n$ be the
discretization step for every integer $n>0$. Set $X_n(0)=x$ and for $kT/n< t\leqslant(k+1)T/n$, $X_n(t)$ is inductively defined by
\begin{eqnarray*}
X_n(t)=X_n(\frac{kT}{n})+\int_{\frac{kT}{n}}^tb(\frac{kT}{n},X_n(\cdot\wedge\frac{kT}{n}))ds+\int_{\frac{kT}{n}}^t
\sigma(\frac{kT}{n},X_n(\cdot\wedge\frac{kT}{n}))dW(s).
\end{eqnarray*}

First we need the following results. We shall make use of the following assumptions.
\begin{itemize}
\item[(A.I)~]
$\sigma:[0,T]\times C([0,T];{\mathbb R}^d)\rightarrow {\mathbb R}^d\otimes{\mathbb R}^m$ satisfies the condition $(A_{1+k,{\mathbb R}^d,{\mathbb R}^d\otimes{\mathbb R}^m})$ and $b:[0,T]\times C([0,T];{\mathbb R}^d)\rightarrow{\mathbb R}^d$ satisfies the condition $(A_{1+k,{\mathbb R}^d,{\mathbb R}^d})$;
\item[(A.II)]
$\sigma$ is bounded and uniformly nondegenerate, i.e., there exists a constant $c>0$ such that $\sigma(t,\varphi)\cdot\sigma^*(t,\varphi)\geqslant c\cdot I$, where $(\sigma(t,\varphi)\cdot\sigma^*(t,\varphi))^{ij}=\sum_{k=1}^r\sigma_k^i(t,\varphi)\sigma_k^j(t,\varphi)$, $i,j=1,\cdots,d$, for all $t\in[0,T]$ and $\varphi\in C([0,T];{\mathbb R}^d)$.
\end{itemize}

\begin{theorem}\label{eulerscheme}
Suppose that the coefficients $(\sigma,b)$ of SDE (\ref{SDE2}) satisfy (A.I) and (A.II). Then for any $p>1$, we have
\begin{eqnarray}\label{e1}
\|X(\cdot,x)\|_{k,p,T;{\mathbb R}^d}\vee\sup_{n\geqslant 1}\|X_n(\cdot,x)\|_{k,p,T;{\mathbb R}^d}<\infty
\end{eqnarray}
and
\begin{eqnarray}\label{e2}
\|X_n(\cdot,x)-X(\cdot,x)\|_{k,p,T;{\mathbb R}^d}=O(n^{-1/2}),
\end{eqnarray}
where the norm $\|\cdot\|_{k,p,T;{\mathbb R}^d}$ is defined in (\ref{soblev norm}).
Furthermore, we also have
\begin{eqnarray}\label{e3}
\|(\det\Sigma_{X(t,x)})^{-1}\|_p&<&\infty
\end{eqnarray}
and
\begin{eqnarray}\label{e4}
\sup_{n\geqslant 1}\|(\det\Sigma_{X_n(t,x)})^{-1}\|_p&<&\infty
\end{eqnarray}
for every $1<p<\infty$, $t\in[0,T]$.
\end{theorem}

Now we give a proof of Theorem \ref{eulerscheme}. Before proceeding, we prepare some propositions. In what follows we denote by $C$ a generic constant which can be different from one formula to another. For convenience we set $k_n(t)=k$ if $kT/n\leqslant t<(k+1)T/n$, and we denote $\eta_n(t)=k_n(t)T/n$.
\begin{proposition}\label{estimate1}
Suppose that the coefficients $(\sigma,b)$ of SDE (\ref{SDE2}) satisfy the following conditions: there exists $C_0<\infty$ such that
\begin{eqnarray}\label{assumption1}
\|\sigma(t,\phi)-\sigma(t,\psi)\|_{HS({\mathbb R}^d;{\mathbb R}^m)}\vee\|b(t,\phi)-b(t,\psi)\|_{{\mathbb R}^d}
\leqslant C_0\|\phi-\psi\|_{C([0,t];{\mathbb R}^d)},
\end{eqnarray}
\begin{eqnarray}\label{assumption3}
\|\sigma(t,\phi)-\sigma(s,\phi)\|_{HS({\mathbb R}^d;{\mathbb R}^m)}\vee\|b(t,\phi)-b(s,\phi)\|_{{\mathbb R}^d}
\leqslant C_0|t-s|^{1/2}(1+\|\phi\|_{C([0,T];{\mathbb R}^d)})
\end{eqnarray}
and
\begin{eqnarray}\label{assumption2}
\|\sigma(t,\phi)\|_{HS({\mathbb R}^d;{\mathbb R}^m)}\vee\|b(t,\phi)\|_{{\mathbb R}^d}\leqslant C_0(1+\|\phi|\|_{C([0,t];{\mathbb R}^d)})
\end{eqnarray}
for every $t,s\in[0,T]$ and $\phi,\psi\in C([0,T];{\mathbb R}^d)$. Then we have
\begin{eqnarray*}
\sup_{n\geqslant 1}E[\sup_{0\leqslant s\leqslant t}\|X_n(s)\|^{2p}]<\infty
\end{eqnarray*}
for all $p>1$ and $t\in[0,T]$. Moreover, for every $p>1$ there exists a constant $C=C(p)<\infty$ such that
\begin{eqnarray*}
\sup_{0\leqslant k\leqslant n-1}E[\sup_{\frac{kT}{n}\leqslant s<\frac{(k+1)T}{n}}\|X_n(s)-X_n(\eta_n(s))\|^{2p}]\leqslant Cn^{-p}
\end{eqnarray*}
for all $n$.
\end{proposition}

\begin{proof}
It is well known that the conditions (\ref{assumption1}) and (\ref{assumption2}) ensure the existence and uniqueness of the solution of SDE (\ref{SDE2}). Since
\begin{eqnarray*}
X_n(t)=X_n(\eta_n(t))+\int_{\eta_n(t)}^tb({\eta_n(s)},X_n(\cdot\wedge\eta_n(s)))ds+\int_{\eta_n(t)}^t
\sigma(\eta_n(s),X_n(\cdot\wedge\eta_n(s)))dW(s),
\end{eqnarray*}
we have
\begin{eqnarray*}
X_n(t)=x+\int_0^tb(\eta_n(s),X_n(\cdot\wedge\eta_n(s)))ds+\int_0^t\sigma(\eta_n(s),X_n(\cdot\wedge\eta_n(s)))dW(s).
\end{eqnarray*}
Consequently
\begin{eqnarray*}
E[\sup_{0\leqslant s\leqslant t}\|X_n(s)\|^{2p}]&\leqslant& 2^{2p-1}E[\sup_{0\leqslant s\leqslant t}\|\int_0^sb(\eta_n(u),X_n(\cdot\wedge\eta_n(u)))du\|^{2p}]\\
&&{}+2^{2p-1}E[\sup_{0\leqslant s\leqslant t}\|\int_0^s\sigma(\eta_n(u),X_n(\cdot\wedge\eta_n(u)))dW(u)\|^{2p}]\\
&\leqslant& 2^{2p-1}t^{2p-1}C_{2p}C_0\int_0^tE[(1+\sup_{0\leqslant u\leqslant s}\|X_n(u)\|)^{2p}]ds\\
&&{}+2^{2p-1}t^{p-1}C_{2p}C_0\int_0^tE[(1+\sup_{0\leqslant u\leqslant s}\|X_n(u)\|)^{2p}]ds.
\end{eqnarray*}
Thus by Gronwall's lemma, the proof of the first conclusion is completed. For the second conclusion, we proceed as follows. Since
\begin{eqnarray*}
X_n(t)-X_n(\eta_n(t))=\int_{\eta_n(t)}^tb(\eta_n(s),X_n(\cdot\wedge\eta_n(s)))ds+\int_{\eta_n(t)}^t
\sigma(\eta_n(s),X_n(\cdot\wedge\eta_n(s)))dW(s),
\end{eqnarray*}
for fixed $k$, where $k=0,1,\cdots,n-1$, using BDG's inequality and (\ref{assumption2}), we have
\begin{eqnarray*}
&&E[\sup_{\frac{kT}{n}\leqslant s<\frac{(k+1)T}{n}}\|X_n(s)-X_n(\eta_n(s))\|^{2p}]\\
&\leqslant& 2^{2p-1}E[\sup_{\frac{kT}{n}\leqslant s<\frac{(k+1)T}{n}}\|\int_{\frac{kT}{n}}^sb(\eta_n(u),X_n(\cdot\wedge\eta_n(u)))du\|^{2p}]\\
&&{}+2^{2p-1}E[\sup_{\frac{kT}{n}\leqslant s<\frac{(k+1)T}{n}}\|\int_{\frac{kT}{n}}^s\sigma(\eta_n(u),
X_n(\cdot\wedge\eta_n(u)))dW(u)\|^{2p}]\\
&\leqslant& 2^{2p-1}n^{-2p+1}\int_{\frac{kT}{n}}^{\frac{(k+1)T}{n}}E[\|b(\eta_n(s),X_n(\cdot\wedge\eta_n(s)))\|^{2p}]ds\\
&&{}+2^{2p-1}n^{-p+1}C_{2p}\int_{\frac{kT}{n}}^{\frac{(k+1)T}{n}}E[\|\sigma(\eta_n(s),X_n(\cdot\wedge\eta_n(s)))\|^{2p}]ds\\
&\leqslant& C_1n^{-2p}+C_2n^{-p},
\end{eqnarray*}
and the proof is completed.
\end{proof}

\begin{proposition}\label{strong approximation}
Suppose that the coefficients $(\sigma,b)$ of SDE (\ref{SDE2}) satisfy the conditions of Proposition \ref{estimate1}. Then we have
\begin{eqnarray*}
(E[\|X_n-X\|_{C([0,t];{\mathbb R}^d)}^p])^{1/p}=O(n^{-1/2})
\end{eqnarray*}
for all $p>1$ and $t\in[0,T]$.
\end{proposition}

\begin{proof}
\begin{eqnarray*}
&&X_n(t)-X(t)\\
&=&X_n(\eta_n(t))-X(\eta_n(t))+\int_{\eta_n(t)}^tb(\eta_n(s),X_n(\cdot\wedge\eta_n(s)))ds-\int_{\eta_n(t)}^tb(s,X(\cdot))ds\\
&&{}+\int_{\eta_n(t)}^t\sigma(\eta_n(s),X_n(\cdot\wedge\eta_n(s)))dW(s)-\int_{\eta_n(t)}^t\sigma(s,X(\cdot))dW(s)\\
&=&X_n(\eta_n(t))-X(\eta_n(t))+\int_{\eta_n(t)}^t(b(s,X_n(\cdot))-b(s,X(\cdot)))ds\\
&&{}+\int_{\eta_n(t)}^t(\sigma(s,X_n(\cdot))-\sigma(s,X(\cdot)))dW(s)\\
&&{}+\Lambda_{\eta_n(t),t},
\end{eqnarray*}
where
\begin{eqnarray*}
\Lambda_{\eta_n(t),t}&:=&\int_{\eta_n(t)}^t(b(\eta_n(s),X_n(\cdot\wedge\eta_n(s)))-b(s,X_n(\cdot\wedge\eta_n(s))))ds\\
&&{}+\int_{\eta_n(t)}^t(b(s,X_n(\cdot\wedge\eta_n(s)))-b(s,X_n(\cdot)))ds\\
&&{}+\int_{\eta_n(t)}^t(\sigma(\eta_n(s),X_n(\cdot\wedge\eta_n(s)))-\sigma(s,X_n(\cdot\wedge\eta_n(s))))dW(s)\\
&&{}+\int_{\eta_n(t)}^t(\sigma(s,X_n(\cdot\wedge\eta_n(s)))-\sigma(s,X_n(\cdot)))dW(s).
\end{eqnarray*}
Repeatedly using the above formula, we obtain
\begin{eqnarray*}
X_n(t)-X(t)&=&\int_0^t(b(s,X_n(\cdot))-b(s,X(\cdot)))ds\\
&&{}+\int_0^t(\sigma(s,X_n(\cdot))-\sigma(s,X(\cdot)))dW(s)\\
&&{}+\sum_{k=0}^{k_n(t)}\Lambda_{\frac{kT}{n},\frac{(k+1)T}{n}\wedge t}.
\end{eqnarray*}
Then using BDG's inequality and (\ref{assumption1}), we have
\begin{eqnarray*}
&&E[\|X_n-X\|_{C([0,t];{\mathbb R}^d)}^{2p}]\\
&\leqslant&C_1\int_0^tE[\|X_n-X\|_{C([0,s];{\mathbb R}^d)}^{2p}]ds+C_2E[\sup_{0\leqslant s\leqslant t}\|\sum_{k=0}^{k_n(s)}
\Lambda_{\frac{kT}{n},\frac{(k+1)T}{n}\wedge s}\|^{2p}].
\end{eqnarray*}
For the last term we proceed as follows. First we have
\begin{eqnarray*}
E[\sup_{0\leqslant s\leqslant t}\|\sum_{k=0}^{k_n(s)}\Lambda_{\frac{kT}{n},\frac{(k+1)T}{n}\wedge s}\|^{2p}]
\leqslant C(\Lambda_1+\Lambda_2+\Lambda_3+\Lambda_4),
\end{eqnarray*}
where
\begin{eqnarray*}
\Lambda_1&=&E[\sup_{0\leqslant s\leqslant t}\|\int_0^s(b(\eta_n(u),X_n(\cdot\wedge\eta_n(u)))
-b(u,X_n(\cdot\wedge\eta_n(u))))du\|^{2p}],\\
\Lambda_2&=&E[\sup_{0\leqslant s\leqslant t}\|\int_0^s(b(u,X_n(\cdot\wedge\eta_n(u)))-b(u,X_n(\cdot)))du\|^{2p}],\\
\Lambda_3&=&E[\sup_{0\leqslant s\leqslant t}\|\int_0^s(\sigma(\eta_n(u),X_n(\cdot\wedge\eta_n(u)))
-\sigma(u,X_n(\cdot\wedge\eta_n(u))))dW(u)\|^{2p}],\\
\Lambda_4&=&E[\sup_{0\leqslant s\leqslant t}\|\int_0^s(\sigma(u,X_n(\cdot\wedge\eta_n(u)))-\sigma(u,X_n(\cdot)))dW(u)\|^{2p}].\\
\end{eqnarray*}
In view of BDG's inequality and (\ref{assumption3}), the terms $\Lambda_1$ and $\Lambda_3$ are dominated by $Cn^{-p}$. Next we estimate the term $\Lambda_4$. In view of BDG's inequality and (\ref{assumption1}), we have
\begin{eqnarray*}
\Lambda_4&=&E[\sup_{0\leqslant s\leqslant t}\|\int_0^s(\sigma(u,X_n(\cdot\wedge\eta_n(u)))-\sigma(u,X_n(\cdot)))dW(u)\|^{2p}]\\
&\leqslant& C_{2p}t^{p-1}\int_0^tE[\|\sigma(s,X_n(\cdot\wedge\eta_n(s)))-\sigma(s,X_n(\cdot))\|^{2p}]ds\\
&\leqslant& C_{2p}C_0t^{p-1}\int_0^tE[\sup_{0\leqslant u\leqslant s}\|X_n(u\wedge\eta_n(s))-X_n(u)\|^{2p}]ds.\\
\end{eqnarray*}
By Proposition \ref{estimate1}, the above expression is dominated by $C'n^{-p}$, where $C'=C_{2p}C_0Ct^p$. The term $\Lambda_2$ can be estimated similarly. Thus combining  the above estimates, we have
\begin{eqnarray*}
E[\|X_n-X\|_{C([0,t];{\mathbb R}^d)}^{2p}]\leqslant C_1\int_0^t E[\|X_n-X\|_{C([0,s];{\mathbb R}^d)}^{2p}]ds+C_2n^{-p}+C_3n^{-2p}.
\end{eqnarray*}
Then the proof is completed by the Gronwall's lemma.
\end{proof}

We denote by $D_t^jF,t\in[0,T],j=1,\cdots,m$, the derivative of a random variable $F$ as an element of
$\mathbb{L}^2([0,T]\times B;{\mathbb R}^m)\cong \mathbb{L}^2(B;{\mathbb H})$. Here the separable Hilbert space ${\mathbb H}$ is an ${\mathbb L}^2$ space of the form ${\mathbb H}={\mathbb L}^2([0,T],{\mathbb R}^m)$. Similarly we denote by $D_{t_1,\cdots,t_n}^{j_1,\cdots,j_n}F$ the $n$-th derivative of $F$. Before the proof of our main proposition, we need the following proposition.
\begin{proposition}\label{estimate2}
Suppose that the coefficients $(\sigma,b)$ of SDE (\ref{SDE2}) satisfy the assumption (A.I), then for any $p>1$, we have
\begin{eqnarray*}
\|X(\cdot,x)\|_{k,p,T;{\mathbb R}^d}\vee\sup_{n\geqslant 1}\|X_n(\cdot,x)\|_{k,p,T;{\mathbb R}^d}<\infty.
\end{eqnarray*}
\end{proposition}

\begin{proof}
We only give a proof of the case of the first order derivative, and the cases of higher order derivatives can be given in a similar way. By results in Kusuoka and Stroock \cite{KS}, we have
\begin{eqnarray*}
D_r^jX_n^i(t)&=&D_r^jX_n^i(\eta_n(t))+\sigma_j^i(\eta_n(r),X_n(\cdot\wedge\eta_n(r)))1_{(\eta_n(t),t]}(r)\\
&&{}+\int_{\eta_n(t)}^t (b^i)^{\alpha_k}(\eta_n(s),X_n(\cdot\wedge\eta_n(s)))D_r^jX_n^k(*\wedge\eta_n(s))ds\\
&&{}+\int_{\eta_n(t)}^t (\sigma_l^i)^{\alpha_k}(\eta_n(s),X_n(\cdot\wedge\eta_n(s)))D_r^jX_n^k(*\wedge\eta_n(s))dW(s)^l,
\end{eqnarray*}
where $(b^{i})^{\alpha_k}$ and $(\sigma_l^i)^{\alpha_k}$ are defined in (\ref{high order derivative}),
$\alpha_k=(0,\cdots,0,1,0,\cdots,0)$.
Repeatedly using the above formula, we have
\begin{eqnarray*}
D_r^jX_n^i(t)&=&\sum_{k=0}^{k_n(t)}\sigma_j^i(\eta_n(r),X_n(\cdot\wedge\eta_n(r)))1_{(\frac{kT}{n},\frac{(k+1)T}{n}]}(r)\\
&&{}+\int_0^t (b^i)^{\alpha_k}(\eta_n(s),X_n(\cdot\wedge\eta_n(s)))D_r^jX_n^k(*\wedge\eta_n(s))ds\\
&&{}+\int_0^t (\sigma_l^i)^{\alpha_k}(\eta_n(s),X_n(\cdot\wedge\eta_n(s)))D_r^jX_n^k(*\wedge\eta_n(s))dW(s)^l.
\end{eqnarray*}
Thus using BDG's inequality and (\ref{a1}), we obtain
\begin{eqnarray*}
\sup_{n\geqslant 1}\sup_{0\leqslant r\leqslant T}E[\sup_{0\leqslant s\leqslant t}\|D_r^jX_n(s)\|^{2p}]<\infty
\end{eqnarray*}
for all $p>1$, $t\in[0,T]$ and $j=1,\cdots,m$. Similarly, we also have
\begin{eqnarray*}
\sup_{0\leqslant r\leqslant T}E[\sup_{0\leqslant s\leqslant t}\|D_r^jX(s)\|^{2p}]<\infty
\end{eqnarray*}
for all $p>1$, $t\in[0,T]$ and $j=1,\cdots,m$, and the proof is thus established.
\end{proof}

\begin{proposition}\label{estimate3}
Suppose that the coefficients $(\sigma,b)$ of SDE (\ref{SDE2}) satisfy the assumption (A.I), then for any $p>1$, we have
\begin{eqnarray*}
\|X_n(\cdot,x)-X(\cdot,x)\|_{k,p,T;{\mathbb R}^d}=O(n^{-1/2}).
\end{eqnarray*}
\end{proposition}

\begin{proof}
First by results in Kusuoka and Stroock \cite{KS}, we have
\begin{eqnarray*}
D_r^jX^i(t)&=&D_r^jX^i(\eta_n(t))+\sigma_j^i(r,X(\cdot))1_{(\eta_n(t),t]}(r)\\
&&{}+\int_{\eta_n(t)}^t (b^i)^{\alpha_k}(s,X(\cdot))D_r^jX^k(*)ds+\int_{\eta_n(t)}^t (\sigma_l^i)^{\alpha_k}(s,X(\cdot))
D_r^jX^k(*)dW(s)^l
\end{eqnarray*}
and
\begin{eqnarray*}
D_r^jX_n^i(t)&=&D_r^jX_n^i(\eta_n(t))+\sigma_j^i(\eta_n(r),X_n(\cdot\wedge\eta_n(r)))1_{(\eta_n(t),t]}(r)\\
&&{}+\int_{\eta_n(t)}^t (b^i)^{\alpha_k}(\eta_n(s),X_n(\cdot\wedge\eta_n(s)))D_r^jX_n^k(*\wedge\eta_n(s))ds\\
&&{}+\int_{\eta_n(t)}^t (\sigma_l^i)^{\alpha_k}(\eta_n(s),X_n(\cdot\wedge\eta_n(s)))D_r^jX_n^k(*\wedge\eta_n(s))dW(s)^l.
\end{eqnarray*}
Thus we have
\begin{eqnarray*}
D_r^jX_n^i(t)-D_r^jX^i(t)&=&D_r^jX_n^i(\eta_n(t))-D_r^jX^i(\eta_n(t))\\
&&{}+\sigma_j^i(\eta_n(r),X_n(\cdot\wedge\eta_n(r)))1_{(\eta_n(t),t]}(r)-\sigma_j^i(r,X(\cdot))1_{(\eta_n(t),t]}(r)\\
&&{}+\int_{\eta_n(t)}^t (b^i)^{\alpha_k}(\eta_n(s),X_n(\cdot\wedge\eta_n(s)))D_r^jX_n^k(*\wedge\eta_n(s))ds\\
&&{}-\int_{\eta_n(t)}^t (b^i)^{\alpha_k}(s,X(\cdot))D_r^jX^k(*)ds\\
&&{}+\int_{\eta_n(t)}^t (\sigma_l^i)^{\alpha_k}(\eta_n(s),X_n(\cdot\wedge\eta_n(s)))D_r^jX_n^k(*\wedge\eta_n(s))dW(s)^l\\
&&{}-\int_{\eta_n(t)}^t (\sigma_l^i)^{\alpha_k}(s,X(\cdot))D_r^jX^k(*)dW(s)^l\\
&=&D_r^jX_n^i(\eta_n(t))-D_r^jX^i(\eta_n(t))\\
&&{}+\int_{\eta_n(t)}^t (b^i)^{\alpha_k}(s,X(\cdot))(D_r^jX_n^k(*)-D_r^jX^k(*))ds\\
&&{}+\int_{\eta_n(t)}^t (\sigma_l^i)^{\alpha_k}(s,X(\cdot))(D_r^jX_n^k(*)-D_r^jX^k(*))dW(s)^l\\
&&{}+\Lambda_{\eta_n(t),t}(r),
\end{eqnarray*}
where $\Lambda_{\eta_n(t),t}(r)$ denotes the remaining terms. Repeatedly using the above formula, we obtain
\begin{eqnarray*}
D_r^jX_n^i(t)-D_r^jX^i(t)&=&\int_0^t (b^i)^{\alpha_k}(s,X(\cdot))(D_r^jX_n^k(*)-D_r^jX^k(*))ds\\
&&{}+\int_0^t (\sigma_l^i)^{\alpha_k}(s,X(\cdot))(D_r^jX_n^k(*)-D_r^jX^k(*))dW(s)^l\\
&&{}+\sum_{k=0}^{k_n(t)}\Lambda_{\frac{kT}{n},\frac{(k+1)T}{n}\wedge t}(r)\\
&=&\int_0^t (b^i)^{\alpha_k}(s,X(\cdot))(D_r^jX_n^k(*)-D_r^jX^k(*))ds\\
&&{}+\int_0^t (\sigma_l^i)^{\alpha_k}(s,X(\cdot))(D_r^jX_n^k(*)-D_r^jX^k(*))dW(s)^l\\
&&{}+\sum_{m=1}^9\Lambda_m(t,r),
\end{eqnarray*}
where
\begin{eqnarray*}
\Lambda_1(t,r)&=&\sum_{k=0}^{k_n(t)}[\sigma_j^i(\eta_n(r),X_n(\cdot\wedge\eta_n(r)))1_{(\frac{kT}{n}, \frac{(k+1)T}{n}\wedge t]}
(r)-\sigma_j^i(r,X(\cdot))1_{(\frac{kT}{n}, \frac{(k+1)T}{n}\wedge t]}(r)],\\
\Lambda_2(t,r)&=&\int_0^t((b^i)^{\alpha_k}(\eta_n(s),X_n(\cdot\wedge\eta_n(s)))-(b^i)^{\alpha_k}(s,X_n(\cdot\wedge\eta_n(s))))
D_r^jX_n^k(*\wedge\eta_n(s))ds,\\
\Lambda_3(t,r)&=&\int_0^t(b^i)^{\alpha_k}(s,X_n(\cdot\wedge\eta_n(s)))(D_r^jX_n^k(*\wedge\eta_n(s))-D_r^jX_n^k(*))ds,\\
\Lambda_4(t,r)&=&\int_0^t((b^i)^{\alpha_k}(s,X_n(\cdot\wedge\eta_n(s)))-(b^i)^{\alpha_k}(s,X_n(\cdot)))D_r^jX_n^k(*)ds,\\
\Lambda_5(t,r)&=&\int_0^t((b^i)^{\alpha_k}(s,X_n(\cdot))- (b^i)^{\alpha_k}(s,X(\cdot)))D_r^jX_n^k(*)ds,\\
\Lambda_6(t,r)&=&\int_0^t((\sigma_l^i)^{\alpha_k}(\eta_n(s),X_n(\cdot\wedge\eta_n(s)))-(\sigma_l^i)^{\alpha_k}(s,X_n
(\cdot\wedge\eta_n(s))))D_r^jX_n^k(*\wedge\eta_n(s))dW(s)^l,\\
\Lambda_7(t,r)&=&\int_0^t(\sigma_l^i)^{\alpha_k}(s,X_n(\cdot\wedge\eta_n(s)))
(D_r^jX_n^k(*\wedge\eta_n(s))-D_r^jX_n^k(*))dW(s)^l,\\
\Lambda_8(t,r)&=&\int_0^t((\sigma_l^i)^{\alpha_k}(s,X_n(\cdot\wedge\eta_n(s)))
-(\sigma_l^i)^{\alpha_k}(s,X_n(\cdot)))D_r^jX_n^k(*)dW(s)^l,\\
\Lambda_9(t,r)&=&\int_0^t((\sigma_l^i)^{\alpha_k}(s,X_n(\cdot))- (\sigma_l^i)^{\alpha_k}(s,X(\cdot)))D_r^jX_n^k(*)dW(s)^l.
\end{eqnarray*}
Then by BDG's inequality for Hilbert space valued stochastic integrals (cf. \cite[Lemma 2.1]{KS}) and Assumption (A.I), we have
\begin{eqnarray*}
&&E[\sup_{0\leqslant t\leqslant T}\|DX_n^i(t)-DX^i(t)\|_{\mathbb H}^{2p}]\\
&=&E[\sup_{0\leqslant t\leqslant T}(\int_0^T\|D_rX_n^i(t)-D_rX^i(t)\|^2dr)^p]\\
&\leqslant&3^{2p-1}E[\sup_{0\leqslant t\leqslant T}(\int_0^T(\int_0^t (b^i)^{\alpha_k}(s,X(\cdot))(D_r^jX_n^k(*)-D_r^jX^k(*))ds)^2dr)^p]\\
&&{}+3^{2p-1}E[\sup_{0\leqslant t\leqslant T}(\int_0^T(\int_0^t (\sigma_l^i)^{\alpha_k}(s,X(\cdot))(D_r^jX_n^k(*)
-D_r^jX^k(*))dW(s)^l)^2dr)^p]\\
&&{}+3^{2p-1}E[\sup_{0\leqslant t\leqslant T}(\int_0^T(\sum_{m=1}^9\Lambda_m(t,r))^2dr)^p]\\
&\leqslant& C_1 \int_0^TE[\sup_{0\leqslant s\leqslant t}\|DX_n(s)-DX(s)\|_{{\mathbb H}
\otimes\mathbb R^d}^{2p}]dt+C_2\sum_{m=1}^9E[\sup_{0\leqslant t\leqslant T}
(\int_0^T(\Lambda_m(t,r))^2dr)^p].
\end{eqnarray*}
Let
\begin{eqnarray*}
\Xi_m:=E[\sup_{0\leqslant t\leqslant T}(\int_0^T(\Lambda_m(t,r))^2dr)^p],\quad m=1,\cdots,9.
\end{eqnarray*}
Then to complete the proof, we just have to show each of nine terms $\Xi_m$ are dominated by $Cn^{-p}$. We first observe that
\begin{eqnarray*}
&&\int_0^T(\sum_{k=0}^{k_n(t)}[\sigma_j^i(\eta_n(r),X_n(\cdot\wedge\eta_n(r)))1_{(\frac{kT}{n}, \frac{(k+1)T}{n}\wedge t]}(r)
-\sigma_j^i(r,X(\cdot))1_{(\frac{kT}{n}, \frac{(k+1)T}{n}\wedge t]}(r)])^2dr\\
&=&\int_0^t(\sigma_j^i(\eta_n(r),X_n(\cdot\wedge\eta_n(r))) -\sigma_j^i(r,X(\cdot)))^2dr\\
&\leqslant&3\int_0^t(\sigma_j^i(\eta_n(r),X_n(\cdot\wedge\eta_n(r))) -\sigma_j^i(r,X_n(\cdot\wedge\eta_n(r))))^2dr\\
&&{}+3\int_0^t(\sigma_j^i(r,X_n(\cdot\wedge\eta_n(r)))-\sigma_j^i(r,X_n(\cdot)))^2dr\\
&&{}+3\int_0^t(\sigma_j^i(r,X_n(\cdot)) -\sigma_j^i(r,X(\cdot)))^2dr\\
&\leqslant&3TC_0^2(1+\|X_n(\cdot)\|_{C([0,T];{\mathbb R}^d)})^2n^{-1}+3C_0^2\int_0^t\sup_{0\leqslant u\leqslant r}|X_n(u\wedge\eta_n(r))-X_n(u)|^2dr\\
&&{}+3C_0^2\int_0^t\sup_{0\leqslant u\leqslant r}|X_n(u)-X(u)|^2dr.
\end{eqnarray*}
Then in view of Proposition \ref{estimate1} and Proposition \ref{strong approximation}, we have
\begin{eqnarray*}
\Xi_1&=&E[\sup_{0\leqslant t\leqslant T}(\int_0^T(\sum_{k=0}^{k_n(t)}[\sigma_j^i(\eta_n(r),X_n(\cdot\wedge\eta_n(r)))
1_{(\frac{kT}{n}, \frac{(k+1)T}{n}\wedge t]}(r)\\
&&{}-\sigma_j^i(r,X(\cdot))1_{(\frac{kT}{n},
\frac{(k+1)T}{n}\wedge t]}(r)])^2dr)^p]\\
&\leqslant&3^{p-1}3^pT^pC_0^{2p}E[(1+\|X_n(\cdot)\|_{C([0,T];{\mathbb R}^d)})^{2p}]n^{-p}\\
&&{}+3^{p-1}3^pT^{p-1}C_0^{2p}\int_0^TE[\sup_{0\leqslant u\leqslant r}|X_n(u\wedge\eta_n(r))-X_n(u)|^{2p}]dr\\
&&{}+3^{p-1}3^pT^{p-1}C_0^{2p}\int_0^TE[\sup_{0\leqslant u\leqslant r}|X_n(u)-X(u)|^{2p}]dr\\
&\leqslant&3^{p-1}3^pT^pC_0^{2p}Cn^{-p}+3^{p-1}3^pT^pC_0^{2p}Cn^{-p}+3^{p-1}3^pT^pC_0^{2p}Cn^{-p}\\
&\leqslant&Cn^{-p}.
\end{eqnarray*}
Next we estimate the terms $\Xi_2$, $\Xi_4$, $\Xi_5$, $\Xi_6$, $\Xi_8$ and $\Xi_9$. Using BDG's inequality for Hilbert space valued stochastic integrals, (\ref{a1}), Proposition \ref{estimate1} and Proposition \ref{estimate2},
we obtain
\begin{eqnarray*}
\Xi_8&=&E[\sup_{0\leqslant t\leqslant T}(\int_0^T(\int_0^t((\sigma_l^i)^{\alpha_k}(s,X_n(\cdot\wedge\eta_n(s)))
-(\sigma_l^i)^{\alpha_k}(s,X_n(\cdot)))D_r^jX_n^k(*)dW(s)^l)^2dr)^p]\\
&\leqslant& C_{2p}T^{p-1}\int_0^TE[(\int_0^T(((\sigma_l^i)^{\alpha_k}(t,X_n(\cdot\wedge\eta_n(t)))
-(\sigma_l^i)^{\alpha_k}(t,X_n(\cdot)))D_r^jX_n^k(*))^2dr)^p]dt\\
&\leqslant& C_{2p}T^{2p-2}\int_0^T\int_0^TE[(((\sigma_l^i)^{\alpha_k}(t,X_n(\cdot\wedge\eta_n(t)))
-(\sigma_l^i)^{\alpha_k}(t,X_n(\cdot)))D_r^jX_n^k(*))^{2p}]drdt\\
&\leqslant& C_{2p}T^{2p-1}\int_0^TE[\|(\sigma_l^i)^{\alpha_k}(t,X_n(\cdot\wedge\eta_n(t)))
-(\sigma_l^i)^{\alpha_k}(t,X_n(\cdot))\|_{Hom(C([0,t];{\mathbb R}^d);{\mathbb R})}^{2p}]dt\\
&\leqslant& C_{2p}C_1T^{2p-1}\int_0^TE[\sup_{0\leqslant s\leqslant t}\|X_n(s\wedge\eta_n(t))-X_n(s)\|^{2p}]dt\\
&\leqslant& C_{2p}C_1C'T^{2p}n^{-p}.
\end{eqnarray*}
The term $\Xi_4$ can be estimated similarly. Using BDG's inequality for Hilbert space valued stochastic integrals, (\ref{a1}), Proposition \ref{strong approximation} and Proposition \ref{estimate2}, we have
\begin{eqnarray*}
\Xi_9&=&E[\sup_{0\leqslant t\leqslant T}(\int_0^T(\int_0^t((\sigma_l^i)^{\alpha_k}(s,X_n(\cdot))-
(\sigma_l^i)^{\alpha_k}(s,X(\cdot)))D_r^jX_n^k(*)dW(s)^l)^2dr)^p]\\
&\leqslant&C_{2p}T^{p-1}\int_0^TE[(\int_0^T(((\sigma_l^i)^{\alpha_k}(t,X_n(\cdot))-
(\sigma_l^i)^{\alpha_k}(t,X(\cdot)))D_r^jX_n^k(*))^2dr)^p]dt\\
&\leqslant&C_{2p}T^{2p-2}\int_0^T\int_0^TE[(((\sigma_l^i)^{\alpha_k}(t,X_n(\cdot))-
(\sigma_l^i)^{\alpha_k}(t,X(\cdot)))D_r^jX_n^k(*))^{2p}]drdt\\
&\leqslant&C_{2p}T^{2p-1}\int_0^TE[\|(\sigma_l^i)^{\alpha_k}(t,X_n(\cdot))
-(\sigma_l^i)^{\alpha_k}(t,X(\cdot))\|_{Hom(C([0,t];{\mathbb R}^d);{\mathbb R})}^{2p}]dt\\
&\leqslant& C_{2p}C_1T^{2p-1}\int_0^T E[\sup_{0\leqslant s\leqslant t}\|X_n(s)-X(s)\|^{2p}]dt\\
&\leqslant& C_{2p}C_1CT^{2p}n^{-p}.
\end{eqnarray*}
We can apply the similar way to estimate the term $\Xi_5$. In view of BDG's inequality for Hilbert space valued stochastic integrals and (\ref{a2}), we can also obtain that the terms $\Xi_2$ and $\Xi_6$ are dominated by $Cn^{-p}$. Finally we deal with the terms $\Xi_3$ and $\Xi_7$. Since
\begin{eqnarray*}
D_r^jX_n^i(t)-D_r^jX_n^i(\eta_n(t))&=&\sigma_j^i(\eta_n(r),X_n(\cdot\wedge\eta_n(r)))1_{(\eta_n(t), t]}(r)\\
&&{}+\int_{\eta_n(t)}^t(b^i)^{\alpha_k}(\eta_n(s),X_n(\cdot\wedge\eta_n(s)))D_r^jX_n^k(*\wedge\eta_n(s))ds\\
&&{}+\int_{\eta_n(t)}^t(\sigma_l^i)^{\alpha_k}(\eta_n(s),X_n(\cdot\wedge\eta_n(s)))
D_r^jX_n^k(*\wedge\eta_n(s))dW(s)^l,
\end{eqnarray*}
using BDG's inequality for Hilbert space valued stochastic integrals and (\ref{a1}), we have
\begin{eqnarray*}
&&\Xi_7\\
&=&E[\sup_{0\leqslant t\leqslant T}(\int_0^T(\int_0^t(\sigma_l^i)^{\alpha_k}(s,X_n(\cdot\wedge\eta_n(s)))
(D_r^jX_n^k(*\wedge\eta_n(s))-D_r^jX_n^k(*))dW(s)^l)^2dr)^p]\\
&\leqslant& C_{2p}C_1T^{p-1}\int_0^T E[(\int_0^T((\sigma_l^i)^{\alpha_k}(t,X_n(\cdot\wedge \eta_n(t)))
(D_r^jX_n^k(*\wedge\eta_n(t))-D_r^jX_n^k(*)))^2dr)^p]dt\\
&\leqslant &C_{2p}C_1T^{p-1}\int_0^T E[(\int_0^T\|D_r^jX_n^k(*\wedge\eta_n(t))-D_r^jX_n^k(*)\|
_{C([\eta_n(t),t];{\mathbb R}])}^2dr)^p]dt\\
&\leqslant &3^{2p-1}C_{2p}C_1T^{p-1}\int_0^T E[(\int_0^T\sup_{\eta_n(t)\leqslant s\leqslant t}|\sigma_j^k(\eta_n(r),
X_n(\cdot\wedge\eta_n(r)))1_{(\eta_n(s),s]}(r)|^2dr)^p]dt\\
&&{}+3^{2p-1}C_{2p}C_1T^{p-1}\int_0^T E[(\int_0^T\sup_{\eta_n(t)\leqslant s\leqslant t}|\int_{\eta_n(s)}^s(b^k)^{\alpha_i}
(\eta_n(u),X_n(\cdot\wedge\eta_n(u))) \\
&&{\quad }D_r^jX_n^i(*\wedge\eta_n(u))du|^2dr)^p]dt\\
&&{}+3^{2p-1}C_{2p}C_1T^{p-1}\int_0^T E[(\int_0^T\sup_{\eta_n(t)\leqslant s\leqslant t}|\int_{\eta_n(s)}^s(\sigma_l^k)
^{\alpha_i}(\eta_n(u),X_n(\cdot\wedge\eta_n(u)))\\
&& {\quad }D_r^jX_n^i(*\wedge\eta_n(u))dW(u)^l|^2dr)^p]dt\\
&:=&C(\Xi_{71}+\Xi_{72}+\Xi_{73}).
\end{eqnarray*}
For the term $\Xi_{71}$, we proceed as follows. We notice that
\begin{eqnarray*}
&&\int_0^T\sup_{\eta_n(t)\leqslant s\leqslant t}|\sigma_j^k(\eta_n(r),X_n(\cdot\wedge\eta_n(r)))1_{(\eta_n(s), s]}(r)|^2dr\\
&=&\int_0^T[|\sigma_j^k(\eta_n(r),X_n(\cdot\wedge\eta_n(r)))1_{(\eta_n(t), t]}(r)|^2]dr\\
&=&\int_{\eta_n(t)}^t[|\sigma_j^k(\eta_n(r),X_n(\cdot\wedge\eta_n(r)))|^2]dr\\
&\leqslant &C_0^2(1+\sup_{0\leqslant u\leqslant T}|X_n(u\wedge\eta_n(r))|)^2n^{-1}.
\end{eqnarray*}
Then we obtain
\begin{eqnarray*}
\Xi_{71}\leqslant TC_0^{2p}E[(1+\sup_{0\leqslant u\leqslant T}|X_n(u\wedge\eta_n(r))|)^{2p}]n^{-p}\leqslant Cn^{-p}.
\end{eqnarray*}
Now we consider the term $\Xi_{73}$. By BDG's inequality and Assumption (A.I), we have
\begin{eqnarray*}
&&\Xi_{73}\\
&\leqslant&T^{p-1}\int_0^T\int_0^TE[\sup_{\eta_n(t)\leqslant s\leqslant t}|\int_{\eta_n(s)}^s(\sigma_l^k)^{\alpha_i}
(\eta_n(u),X_n(\cdot\wedge\eta_n(u)))D_r^jX_n^i(*\wedge\eta_n(u))dW(u)^l|^{2p}]drdt\\
&\leqslant&C_{2p}T^{p-1}n^{-p+1}\int_0^T\int_0^T\int_{\eta_n(t)}^tE[|(\sigma_l^k)^{\alpha_i}(\eta_n(s),X_n
(\cdot\wedge\eta_n(s)))D_r^jX_n^i(*\wedge\eta_n(s))|^{2p}]dsdrdt\\
&\leqslant&C_{2p}T^{p-1}n^{-p+1}\int_0^T\int_0^T\int_{\eta_n(t)}^tE[\sup_{0\leqslant u\leqslant \eta_n(s)}|D_r^jX_n^i
(u)|^{2p}]dsdrdt\\
&\leqslant& Cn^{-p}.
\end{eqnarray*}
A similar estimate holds for the term $\Xi_{72}$. The term $\Xi_3$ can be estimated similarly. Thus we have shown
\begin{eqnarray*}
E[\sup_{0\leqslant t\leqslant T}\|DX_n(t)-DX(t)\|_{{\mathbb H}\otimes \mathbb R^d}^{2p}]
\leqslant C_1 \int_0^TE[\sup_{0\leqslant s\leqslant t}\|DX_n(s)-DX(s)\|_{{\mathbb H}\otimes \mathbb R^d}^{2p}]dt+C_2n^{-p}.
\end{eqnarray*}
Then by the Gronwall's lemma, we have
\begin{eqnarray*}
E[\sup_{0\leqslant t\leqslant T}\|DX_n(t)-DX(t)\|_{{\mathbb H}\otimes \mathbb R^d}^{2p}]\leqslant Cn^{-p}.
\end{eqnarray*}
For the case of higher order derivatives, we proceed in the same way.
The proof is therefore completed.
\end{proof}

Now we are ready to give the proof of Theorem \ref{eulerscheme}.
\begin{proof}[Proof of Theorem \ref{eulerscheme}]
The conclusions (\ref{e1}) and (\ref{e2}) have been proved in Proposition \ref{estimate2} and Proposition \ref{estimate3}, respectively. So it remains only to prove (\ref{e3}) and (\ref{e4}). Since
\begin{eqnarray*}
X_n(t)=x+\int_0^tb(\eta_n(s),X_n(\cdot\wedge\eta_n(s)))ds+\int_0^t\sigma(\eta_n(s),X_n(\cdot\wedge\eta_n(s)))dW(s),
\end{eqnarray*}
we can rewrite it in the following form:
\begin{eqnarray*}
X_n(t)=x+\int_0^tb_n(s,X_n(\cdot))ds+\int_0^t\sigma_n(s,X_n(\cdot))dW(s),
\end{eqnarray*}
where
\begin{eqnarray*}
b_n(s,X_n(\cdot)):=b(\eta_n(s),X_n(\cdot\wedge\eta_n(s)))~\text{and}~\sigma_n(s,X_n(\cdot)):=\sigma(\eta_n(s),
X_n(\cdot\wedge\eta_n(s))),
\end{eqnarray*}
and the coefficients $(\sigma_n,b_n)$ also satisfy the uniform elliptic condition. Thus by
\cite[Theorem 3.5, Corollary 3.9]{KS}, we obtain the nondegeneracy of $X(T)$ and $X_n(T)$. Therefore we complete the proof.
\end{proof}

We can now state our main results which extend Theorem \ref{eulerscheme} to the case of interpolation spaces. We shall replace the above assumption (A.I) with the following assumption (A.III).
\begin{itemize}
\item[(A.III)~]
$\sigma:[0,T]\times C([0,T];{\mathbb R}^d)\rightarrow{\mathbb R}^d\otimes{\mathbb R}^m$ satisfies the condition $(A_{1+\delta,{\mathbb R}^d,{\mathbb R}^d\otimes{\mathbb R}^m})$ and $b:[0,T]\times C([0,T];{\mathbb R}^d)\rightarrow{\mathbb R}^d$ satisfies the condition $(A_{1+\delta,{\mathbb R}^d,{\mathbb R}^d})$.
\end{itemize}

\begin{theorem}\label{eulerschemefractionalregularity}
Let $\delta'=k+\theta'$, $k\in{\mathbb N}$, $0<\delta'<\delta$ and $k_n=2^{2n}$. Suppose that the coefficients $(\sigma,b)$ of SDE (\ref{SDE2}) satisfy (A.II) and (A.III). Then we have
\begin{eqnarray}\label{e1f}
\|X(\cdot,x)\|_{\delta',p,T;{\mathbb R}^d}\vee\sup_{n\geqslant 1}\|X_n(\cdot,x)\|_{\delta',p,T;{\mathbb R}^d}<\infty
\end{eqnarray}
and
\begin{eqnarray}\label{e2f}
\|X_{k_n}(\cdot,x)-X(\cdot,x)\|_{\delta',p,T;{\mathbb R}^d}=O(2^{-n\theta}),
\end{eqnarray}
where the norm $\|\cdot\|_{\delta',p,T;{\mathbb R}^d}$ is defined in (\ref{fractional regularity soblev norm}).
\end{theorem}

\begin{proof}
For simplicity of notations we denote $X_{k_n}$ by $X_n$. We will assume $0<\delta<1$. The general case can be treated in the same way by considering the SDE satisfied by the Malliavin-Shigekawa gradient of $X$ and $X_n$. Since $\sigma:[0,T]\times C([0,T];{\mathbb R}^d)\rightarrow{\mathbb R}^d\otimes{\mathbb R}^m$ satisfies the condition $(A_{1+\delta,{\mathbb R}^d,{\mathbb R}^d\otimes{\mathbb R}^m})$ and $b:[0,T]\times C([0,T];{\mathbb R}^d)\rightarrow{\mathbb R}^d$ satisfies the condition $(A_{1+\delta,{\mathbb R}^d,{\mathbb R}^d})$, there are $\sigma(t)\in(C_b^1(C([0,t];{\mathbb R}^d);{\mathbb R}^d\otimes{\mathbb R}^m),C_b^2(C([0,t];{\mathbb R}^d);{\mathbb R}^d\otimes{\mathbb R}^m))_{\delta,p}$ and $b(t)\in(C_b^1(C([0,t];{\mathbb R}^d);{\mathbb R}^d),C_b^2(C([0,t];{\mathbb R}^d);{\mathbb R}^d))_{\delta,p}$ such that $\sigma(t,\psi)=\sigma(t)(\psi|_{[0,t]})$ and $b(t,\psi)=b(t)(\psi|_{[0,t]})$ for all $\psi\in C([0,T];{\mathbb R}^d)$. By Remark \ref{discrete}, for each $m$, we can find $\sigma_m(t)$ and $b_m(t)$ such that
\begin{eqnarray}\label{d}
\begin{split}
\|\sigma_m(t)-\sigma(t)\|_{C_b^1(C([0,t];{\mathbb R}^d);{\mathbb R}^d\otimes{\mathbb R}^m)}\leqslant C2^{-m\delta},\quad \|\sigma_m(t)\|_{C_b^2(C([0,t];{\mathbb R}^d);{\mathbb R}^d\otimes{\mathbb R}^m)}\leqslant C2^{m(1-\delta)},\\
\|b_m(t)-b(t)\|_{C_b^1(C([0,t];{\mathbb R}^d);{\mathbb R}^d)}\leqslant C2^{-m\delta},\quad \|b_m(t)\|_{C_b^2(C([0,t];{\mathbb R}^d);{\mathbb R}^d)}\leqslant C2^{m(1-\delta)}.
\end{split}
\end{eqnarray}
Now we consider the following SDEs:
\begin{eqnarray*}
X_{0,m}(t)&=&x+\int_0^tb_{m+n}(s,X_{0,m}(\cdot))ds+\int_0^t\sigma_{m+n}(s,X_{0,m}(\cdot))dW(s),\\
X_{n,m}(t)&=&X_{n,m}(\frac{kT}{2^{2n}})+\int_{\frac{kT}{2^{2n}}}^tb_{m+n}(\frac{kT}{2^{2n}},X_{n,m}(\cdot\wedge\frac{kT}{2^{2n}}))ds\\
&&+\int_{\frac{kT}{2^{2n}}}^t\sigma_{m+n}(\frac{kT}{2^{2n}},X_{n,m}(\cdot\wedge\frac{kT}{2^{2n}}))dW(s).
\end{eqnarray*}
Now we prove the second conclusion (\ref{e2f}). By Assumption (A.III), from the proof of Theorem \ref{eulerscheme}, we can obtain
\begin{eqnarray*}
\|X_{0,m}(\cdot)-X_{n,m}(\cdot)\|_{1,p,T;{\mathbb R}^d}^p&\leqslant& C(\|b_{m+n}(T)\|_{C_b^2(C([0,T];{\mathbb R}^d);{\mathbb R}^d)}^p\vee
\|\sigma_{m+n}(T)\|_{C_b^2(C([0,T];{\mathbb R}^d);{\mathbb R}^d\otimes{\mathbb R}^m)}^p)\\
&&e^{C(\|b_{m+n}(T)\|_{C_b^1(C([0,T];{\mathbb R}^d);{\mathbb R}^d)}\vee
\|\sigma_{m+n}(T)\|_{C_b^1(C([0,T];{\mathbb R}^d);{\mathbb R}^d\otimes{\mathbb R}^m)})}2^{-np}.
\end{eqnarray*}
Hence by (\ref{d}) we have
\begin{eqnarray*}
\|X_{0,m}(\cdot)-X_{n,m}(\cdot)\|_{1,p,T;{\mathbb R}^d}^p\leqslant C2^{(m+n)(1-\delta)p}2^{-np}.
\end{eqnarray*}
Since
\begin{eqnarray*}
&&X(t)-X_{0,m}(t)\\
&=&\int_0^t(b(s,X(\cdot))-b_{m+n}(s,X_{0,m}(\cdot)))ds+\int_0^t(\sigma(s,X(\cdot))-\sigma_{m+n}(s,X_{0,m}(\cdot)))dW(s)\\
&=&\int_0^t(b(s,X(\cdot))-b(s,X_{0,m}(\cdot)))ds+\int_0^t(\sigma(s,X(\cdot))-\sigma(s,X_{0,m}(\cdot)))dW(s)\\
&&{}+\int_0^t(b(s,X_{0,m}(\cdot))-b_{m+n}(s,X_{0,m}(\cdot)))ds+\int_0^t(\sigma(s,X_{0,m}(\cdot))-\sigma_{m+n}(s,X_{0,m}(\cdot)))dW(s),\\
\end{eqnarray*}
by BDG's inequality, Assumption (A.III) and (\ref{d}), we have
\begin{eqnarray*}
E[\|X(\cdot)-X_{0,m}(\cdot)\|_{C([0,t];{\mathbb R})}^{p}]\leqslant C_1\int_0^tE[\|X(\cdot)-X_{0,m}(\cdot)\|_{C([0,s];{\mathbb R})}^{p}]ds+C_22^{-(m+n)\delta p}.
\end{eqnarray*}
Therefore we deduce by Gronwall's lemma that
\begin{eqnarray*}
\|X(\cdot)-X_{0,m}(\cdot)\|_{0,p,T;{\mathbb R}^d}^p\leqslant C2^{-(m+n)\delta p}.
\end{eqnarray*}
Similarly, we also have
\begin{eqnarray*}
\|X_n(\cdot)-X_{n,m}(\cdot)\|_{0,p,T;{\mathbb R}^d}^p\leqslant C2^{-(m+n)\delta p}.
\end{eqnarray*}
Combining with the above two inequalities we have by Remark \ref{discrete}
\begin{eqnarray*}
\|X_n(\cdot)-X(\cdot)\|_{\delta',p,T;{\mathbb R}^d}^p&\leqslant&C\sum_{m=0}^{\infty}2^{\delta' mp}(\|(X(\cdot)-X_{0,m}(\cdot))-(X_n(\cdot)-X_{n,m}(\cdot))\|_{0,p,T;{\mathbb R}^d}^p\\
&&{}+2^{-mp}\|X_{0,m}(\cdot)-X_{n,m}(\cdot)\|_{1,p,T;{\mathbb R}^d}^p)\\
&\leqslant&C\sum_{m=0}^{\infty}2^{\delta' mp}(2^{-(m+n)\delta p}+2^{-mp}2^{(m+n)(1-\delta)p}2^{-np})\\
&=&C\sum_{m=0}^{\infty}2^{\delta' mp}(2^{-\delta mp}2^{-n\delta p}+2^{-\delta mp}2^{-n\delta p})\\
&\leqslant&C2^{-n\delta p}.
\end{eqnarray*}
Applying the same procedure as used as above, we can get the first conclusion (\ref{e1f}), and we thus complete the proof.
\end{proof}

Combining this with Theorem \ref{comparison theorem}, we can obtain the convergence rate of $\xi_2(x,T,n)$. In order to get this result, we shall make use the following assumption (A.IV).
\begin{itemize}
\item[(A.IV)~]
$\sigma:[0,T]\times C([0,T];{\mathbb R}^d)\rightarrow{\mathbb R}^d\otimes{\mathbb R}^m$ satisfies the condition $(A_{2+\delta,{\mathbb R}^d,{\mathbb R}^d\otimes{\mathbb R}^m})$ and $b:[0,T]\times C([0,T];{\mathbb R}^d)\rightarrow{\mathbb R}^d$ satisfies the condition $(A_{2+\delta,{\mathbb R}^d,{\mathbb R}^d})$.
\end{itemize}
\begin{theorem}\label{main theorem}
Let $k_n=2^{2n}$. Suppose that the coefficients $(\sigma,b)$ of SDE (\ref{SDE2}) satisfy (A.II) and (A.IV). If $0<\beta<\alpha\wedge(1+\delta)-1-d/q$, $1/p+1/q=1$,
$G\in{\mathbb D}_{\alpha}^q$ and $t\in [0,T]$, then we have $p_{X_{k_n}(t),G}\in C^{\beta}({\mathbb R}^d)$, $p_{X(t),G}\in C^{\beta}({\mathbb R}^d)$ and
\begin{eqnarray}
\|p_{X_{k_n}(t),G}-p_{X(t),G}\|_{C^\beta({\mathbb R}^d)}=O(2^{-n\theta}),
\end{eqnarray}
where $p_{X_{k_n}(t),G}$ and $p_{X(t),G}$ are defined by (\ref{density}), i.e.,
\begin{eqnarray*}
p_{X_{k_n}(t),G}(y)=E[G\cdot\delta_y\circ X_{k_n}(t)]=E(G|X_{k_n}(t)=y)p_{X_{k_n}(t)}(y)
\end{eqnarray*}
and
\begin{eqnarray*}
p_{X(t),G}(y)=E[G\cdot\delta_y\circ X(t)]=E(G|X(t)=y)p_{X(t)}(y).
\end{eqnarray*}
In particular, taking $G=\textbf{1}\in{\mathbb D}_{\infty}^{\infty-}$, we conclude that $p_{X_{k_n}(t)}\in C^{\beta}({\mathbb R}^d)$,
$p_{X(t)}\in C^{\beta}({\mathbb R}^d)$ and
\begin{eqnarray}\label{main result}
\|p_{X_{k_n}(t)}-p_{X(t)}\|_{C^\beta({\mathbb R}^d)}=O(2^{-n\theta}).
\end{eqnarray}
Furthermore, if $G_n$, $n=1,2,\cdots$ and $G$ are in ${\mathbb D}^q_{\alpha}$ and
\begin{eqnarray*}
\|G_n-G\|_{\alpha,q}=O(2^{-n\lambda}),
\end{eqnarray*}
then we have
\begin{eqnarray}
\|p_{X_{k_n}(t),G_n}-p_{X(t),G}\|_{C^\beta({\mathbb R}^d)}=O(2^{-n(\lambda\wedge\theta)}).
\end{eqnarray}
\end{theorem}

\begin{proof}
One only needs to simply combine Theorem \ref{eulerschemefractionalregularity} and Theorem \ref{comparison theorem}.
\end{proof}

\begin{remark}
It is worth noting that for some special non-Markovian stochastic differential equations such as stochastic differential delay equations the convergence rate of the density can be improved to a better result. Indeed, Cl\'{e}ment, Kohatsu-Higa and Lamberton \cite{CK-HL} have obtained the convergence rate of the density for some stochastic differential delay equations is $1/n$.
\end{remark}

\begin{Examples}
Consider the following stochastic delay differential equations
\begin{eqnarray*}
x(t)=x+\int_0^tb(x(s-\tau),x(s))ds+\int_0^t\sigma(x(s-\tau),x(s))dW(s),
\end{eqnarray*}
where the coefficients $b$ and $\sigma$ are mapping from ${\mathbb R}^2\rightarrow{\mathbb R}$. Let $T>0$ be a fixed time horizon, and $T/n$ represent the discretization step. Set $x_n(0)=x$, and for $kT/n<t\leqslant(k+1)T/n$, the Euler scheme is defined by
\begin{eqnarray*}
x_n(t)=x_n(\frac{kT}{n})+\int_{\frac{kT}{n}}^tb(x_n(\frac{kT}{n}-\tau),x_n(\frac{kT}{n}))ds+\int_{\frac{kT}{n}}^t\sigma(x_n(\frac{kT}{n}-\tau),x_n(\frac{kT}{n}))dW(s).
\end{eqnarray*}
Let $\delta=k+\theta$, $k\in{\mathbb N}$, $0<\theta\leqslant 1$ and $p>1$. Suppose that the coefficients $(b(x),\sigma(x))$ satisfy the following assumptions: $b,\sigma\in(C_b^{2+k}({\mathbb R}^2;{\mathbb R}),C_b^{3+k}({\mathbb R}^2;{\mathbb R}))_{\theta,p}$; $\sigma$ is bounded and uniformly nondegenerate. Then we can use Theorem \ref{main theorem} to obtain the convergence rate of density of the Euler scheme for stochastic delay differential equations. Similar results holds of course for more general delay equations.
\end{Examples}

\begin{remark}
Finally we point out that on a finite dimensional Euclidean space, say ${\mathbb R}^m$, the space $(C^{k}({\mathbb R}^m), C^{k+1}({\mathbb R}^m))_{\theta,p}$ is very close to $C^{k+\theta}({\mathbb R}^m)$. In fact, it is known that (see \cite{T}) that for every $\epsilon>0$,
\begin{eqnarray*}
C^{k+\theta}({\mathbb R}^m)&=&(C^k({\mathbb R}^m),C^{k+1}({\mathbb R}^m))_{\theta,\infty}\subset (C^k({\mathbb R}^m),C^{k+1}({\mathbb R}^m))_{\theta-\epsilon,p}\\
&&\subset (C^k({\mathbb R}^m),C^{k+1}({\mathbb R}^m))_{\theta-\epsilon,\infty}=C^{k+\theta-\epsilon}({\mathbb R}^m).
\end{eqnarray*}
On an infinite dimensional Banach space $E$, this full chain of inclusions does not hold in general, but we still have the following
\begin{eqnarray*}
(C^k(E),C^{k+1}(E))_{\theta,p}\subset C^{k+\theta}(E),\quad\forall p\in [1,\infty).
\end{eqnarray*}
Therefore the above example shows that we can obtain the convergence rate of density of the Euler scheme for stochastic delay differential equations when the coefficients are in  H\"{o}lder spaces.
\end{remark}

\vskip 1cm

{\bf Acknowledgement}
We thank heartily Professor V. Bally for informing us his recent works and, in particular, for sending us the reference \cite{BC2}, which has helped us to improve greatly the quality of the paper. Also we are very grateful to the anonymous referees and the associated editor for their careful reading of the manuscript and for valuable suggestions and criticisms.

\end{document}